\documentclass[letterpaper,10pt,journal]{IEEEtran}

\usepackage{graphics}
\usepackage{epsfig} 
\usepackage{amsmath}
\usepackage{amssymb}
\usepackage{dsfont}
\usepackage{enumerate}
\usepackage[tight]{subfigure}
\usepackage{multirow}
\usepackage{bm}
\usepackage{tikz}
\usepackage[noadjust]{cite}
\usepackage{amsthm}  

\usepackage{hyperref}

\usepackage[ruled,linesnumbered]{algorithm2e}

 \RestyleAlgo{ruled}

\SetKwInput{KwInit}{Initialise}
 \SetKwRepeat{Do}{do}{while}

\usetikzlibrary{arrows.meta}
\usetikzlibrary{positioning}

\newcommand{\F}{\mathcal{F}}
\newcommand{\Fone}{\mathcal{F}_Z}
\newcommand{\Ftwo}{\mathcal{F}_{Z,NL}}
\newcommand{\Fzero}{\mathcal{F}_{ALL}}

\newcommand{\x}{\mathbf{x}}
\newcommand{\y}{\mathbf{y}}
\newcommand{\z}{\mathbf{z}}
\newcommand{\w}{\mathbf{w}}
\newcommand{\vv}{\mathbf{v}}
\newcommand{\Y}{\mathbf{Y}}

\usepackage{color}

\title{Nonlinear Network Identifiability \linebreak with Full Excitations
}

\author{Renato Vizuete and Julien M. Hendrickx
\thanks{This work was supported by F.R.S.-FNRS via the \emph{KORNET} project, and by the \emph{SIDDARTA} Concerted Research Action (ARC) 
of the Fédération Wallonie-Bruxelles.}
\thanks{R.~Vizuete and J.~M.~Hendrickx are with ICTEAM institute, UCLouvain, B-1348, Louvain-la-Neuve, Belgium. R.~Vizuete is a FNRS Postdoctoral Researcher - CR.
{\tt\small renato.vizueteharo@uclouvain.be},
{\tt\small julien.hendrickx@uclouvain.be}\protect.}
}

\newtheorem{definition}{Definition}
\newtheorem{assumption}{Assumption}
\newtheorem{theorem}{Theorem}
\newtheorem{corollary}{Corollary}
\newtheorem{proposition}{Proposition}
\newtheorem{lemma}{Lemma}
\newtheorem{remark}{Remark}
\newtheorem{example}{Example}

\newcommand{\prt}[1]{\left(#1\right)}

\newcommand{\brc}[1]{\left\{#1\right\}}
\newcommand{\abs}[1]{\left|#1\right|}

\newcommand{\R}{\mathbb R}

\newcommand{\NN}{\mathcal{N}}
\newcommand{\aNN}{\abs{\mathcal{N}_i}}

\newcommand{\zero}{\mathbf{0}}

\begin{document}

\maketitle
\thispagestyle{empty}

%%%%%%%%%%%%%%%%%%%%%%%%%%%%%%%%%%%%%%%%%%%%%%%%%%%%%%%%%%%%%%%%%%%%%%%%%%%%%%%%
\begin{abstract}
We derive conditions for the identifiability of nonlinear networks characterized by additive dynamics at the level of the edges when all the nodes are excited. In contrast to linear systems, we show that the measurement of all sinks is necessary and sufficient for the identifiability of directed acyclic graphs, under the assumption that dynamics are described by twice continuously differentiable functions without constant terms (i.e., $f(0)=0$). But if constant terms are present, then the identifiability is impossible as soon as one node has more than one in-neighbor. In the case of general digraphs that may contain cycles, we consider additively separable functions for the analysis of the identifiability, and we show that the measurement of one node of all the sinks of the condensation digraph is necessary and sufficient. Several examples are added to illustrate the results.
\end{abstract}


%%%%%%%%%%%%%%%%%%%%%%%%%%%%%%%%%%%%%%%%%%%%%%%%%%%%%%%%%%%%%%%%%%%%%%%%%%%%%%%%
\section{Introduction}

Networked systems composed of single entities or subsystems that interact in a network can be found in many domains including biological networks, power networks, social networks, etc \cite{boccaletti2006complex,bullo2024lectures}. In these types of systems, one of the most important characteristics is the network topology, which models the exchange of information among the nodes. However, in many networks, different dynamical systems may also be associated at the level of edges or nodes, which generates a more complex behavior in the system. 

Knowledge of the interactions between the nodes is useful and in many situations can be essential for the analysis, estimation and control of networks \cite{wang2014pinning}. For this reason, the identification of networks has become an important research area where the objective is to identify the network topology and all the dynamics associated with nodes and edges \cite{gonccalves2008necessary,mauroy2017spectral,jahandari2022optimal,verhaegen2022data}. Since in many networks the number of nodes and edges can be relatively large, the appropriate placement of sensors and instruments to perform experiments for the identification process is crucial \cite{ramaswamy2019generalized,bombois2023informativity,kivits2023identifiability}. In this case, when the network topology is available, the notion of \emph{identifiability} plays an important role since it allows us to determine which nodes should be excited and measured to identify the different dynamics in the network. 
The identifiability in networked systems is based on the classical notion of identifiability used in system identification \cite{ljung1999system}, where if each set of local dynamics (edges) leads to a unique global behavior (measurements), then it means that the local dynamics can be distinguished based on the global behavior. Hence, the focus of the identifiability is the analysis of the uniqueness of the local dynamics and not the design of identification procedures. 
The identifiability problem in networks has been extensively studied in the case of linear systems associated with edges where, based on the knowledge of the network topology, conditions for the full measurement or full excitation cases have been provided in \cite{hendrickx2019identifiability,vanwaarde2020necessary}. The case of partial measurement/excitation has been analyzed in several works \cite{legat2020local,legat2021path,bazanella2019network,cheng2023necessary,legat2023combinatorial} where different identifiability conditions have been proposed depending on the scenario, though no necessary and sufficient condition is known. However, most real networks are nonlinear, including systems in important fields of interest like
coupled oscillators \cite{dorfler2014synchronization}, gene regulatory networks \cite{pan2012reconstruction}, biochemical reaction networks \cite{aalto2020gene}, organic reactions \cite{bures2023organic}, social networks \cite{bizyaeva2023nonlinear}, among others. Regarding the identifiability of nonlinear networks, in \cite{verdiere2024identifiability}, the problem has been analyzed in continuous time where an algorithm has been proposed to determine the identifiability of the dynamics based on the knowledge of specific variables over time. However, to the best of our knowledge, no identifiability conditions based on the network topology have been derived for nonlinear networks in discrete time.

The identification of even a single nonlinear system is a challenging task due to the diverse types of nonlinearities and models that can be used to reproduce the behavior of the system \cite{schoukens2019nonlinear}. In fact, the choice of the model might depend on the type of application and some particular models could not be well suited for several systems \cite{janczak2004identification,nelles2020nonlinear,paduart2010identification}. At the level of nonlinear networks, the complexity increases considerably since essentially the identification of nonlinear functions that represent the dynamics of edges or nodes must be performed without disconnecting the network. In this case, identifiability conditions could be even more important than in the linear case since in general the experiments for nonlinear identification are more complex and expensive, so that the knowledge of which nodes should be excited and measured is fundamental. Although the analysis of identifiability is based on all the information that could theoretically be obtained from measurements (even if it would not necessarily be practical to obtain so much information), it can give us a good intuition for the outcome of experiments. If an edge is not identifiable based on such perfect information, then it can certainly not be identified with partial information. On the other hand, if an edge can be identified in the ideal situation, it could be possible that with appropriate experiments we manage to obtain an approximation of the edge.

The class of functions considered in the identifiability of nonlinear networks is important for the derivation of conditions that guarantee the identification. The nonlinear dynamics of real networks can be modeled by functions of different nature like trigonometric in coupled oscillators \cite{dorfler2014synchronization} or non-differentiable in neural networks \cite{aggarwal2018neural}. Clearly, the identifiability should be analyzed in a specific class of functions and an appropriate delimitation of this class is fundamental. If the class of functions is too restrictive, it could be easier to guarantee identifiability but the type of functions could not match real dynamical models. If the class of functions is overly general, identifiability might be very restrictive since many different functions could satisfy the information obtained with measurements. Nevertheless, we will derive strong results for the identifiability of nonlinear networks for very general classes of functions.

A preliminary version of this work was presented in \cite{vizuete2023nonlinear} where identifiability conditions for directed acyclic graphs (DAGs) characterized by nonlinear static interactions were derived in the full excitation case to highlight the differences with linear networks. By contrast, in this work, we derive identifiability conditions in the full excitation case for more complex nonlinear networks where the edges are characterized by nonlinear dynamical systems (i.e., involving memory), and cycles are allowed. In this way, our results can be applied to real networks where the dynamics are one of the main characteristics of the edges and the presence of cycles is common. Finally, we extend the class of analytic functions used in \cite{vizuete2023nonlinear} to continuously differentiable functions and twice continuously differentiable functions.

The remainder of this article is organized as follows.
In Section~\ref{sec:problem_formulation}, we present the dynamical model associated with each node and the notions of identifiability at the level of edges and networks. In Section~\ref{sec:DAG}, we study the identifiability problem in DAGs. For path graphs and trees, we analyze the identifiability problem in a smaller class of functions where each function is continuously differentiable and the static part is removed, and we show that the measurement of the sink is necessary and sufficient for identifiability. In the case of general DAGs, we discard linear functions for the identifiability problem and we prove that the measurement of the sinks is necessary and sufficient for the identifiability of twice continuously differentiable functions. Section~\ref{sec:general_digraphs} focuses on the case of general digraphs where cycles are allowed. By considering that the network is at rest at the initial time, we show that by measuring a node of all the sinks of the condensation digraph, we guarantee the identifiability of any digraph. 
Finally, conclusions and future perspectives for the identifiability of nonlinear networks are exposed in Section~\ref{sec:conclusions}.

\section{Problem formulation}\label{sec:problem_formulation}

\subsection{Notation}\label{sec:notation}

Scalars are denoted by $x\in\R$ and vectors are denoted by $\mathbf{x}\in\R^n$. The vector of an appropriate size constituted of only zeros is denoted by $\mathbf{0}$. 
The value of a general input or output $x$ at the time instant $k$ is denoted by $x^k$. The shift operator acting on a vector $\x=(x_1,\ldots,x_n)$ corresponds to a shift of coordinates and it is denoted by $\tau^p\x=(x_{1+p},\ldots,x_{n+p})$, where we assume that $\x$ is well defined for coordinates $x_i$ with $i>n$ \footnote{In this work, the shift operator will be used to shift coordinates in time.  For instance, for a  vector $\x_1=(x_1^{k-1},x_1^{k-2})$, the result after the application of the shift operator $\tau^1$ will be $\tau^1\x_1=(x_1^{k-2},x_1^{k-3})$. Since the number of all the possible delayed signals that could appear in a vector depends on the network topology, delays of other edge functions and even the time instant of measurement of a node (see Subsection~\ref{sec:model_class}), we assume that the signals are well defined at all time instants considered in the analysis.}.

Next, we introduce the following definitions associated with graphs.

\begin{itemize}
    \item A \emph{source} is a node with no incoming edges, and a \emph{sink} is a node with no outgoing edges.
    \item A node $i$ is an \emph{in-neighbor} of node $j$ if there exists a directed edge from $i$ to $j$ denoted by $(j,i)$, and a node $i$ is an \emph{out-neighbor} of node $j$ if there exists a directed edge from $j$ to $i$ denoted by $(i,j)$.
    \item A \emph{walk} is a series of adjacent directed edges. 
    \item A \emph{path} is a walk that never passes twice through the same node.
    \item The \emph{diameter} of a digraph is the length of the shortest path between the most distant nodes.
    \item A \emph{cycle} or a loop in a digraph is a series of adjacent directed edges in which only the first and last vertices are equal.
    \item A digraph $G$ is \emph{weakly connected} if the undirected version of the digraph (i.e., direction of edges are ignored) is connected.
    \item A \emph{self-loop} is an edge that connects a vertex to itself.
    \item A \emph{path graph} is a digraph whose vertices can be listed in the order $1,\ldots,n$ such that the edges are $(i+1,i)$ with $i=1,\ldots,n-1$.
    \item A \emph{tree} is a digraph that has no loops even if we change the edges directions.
    \item An \emph{arborescence} is a digraph where for a node $u$ called the root and any other node $v$, there is exactly one path.
    \item A \emph{directed acyclic graph (DAG)} is a digraph without loops. (See Fig.~\ref{fig:DAG_volterra}).
    \item An \emph{induced subgraph} $G_S$ of a digraph $G$ is a subgraph formed by a subset $S$ of the vertices of $G$ and all the edges of $G$ that have both endpoints in $S$. (See Fig.~\ref{fig:induced_graph}). 
    \item A \emph{strongly connected component} $H$ of a digraph $G$ is a subgraph that is strongly connected such that any other subgraph of $G$ strictly containing $H$ is not strongly connected. (See Fig.~\ref{fig:general_digraph}).
    \item The \emph{condensation digraph} $C(G)$ of a digraph $G$ is a digraph where each strongly connected component is replaced by a node and all the edges from one strongly connected component to another are replaced by a single edge. The condensation digraph is acyclic. (See Fig.~\ref{fig:general_digraph}).
    \item A \emph{$k$-partite digraph} is a digraph whose vertices can be partitioned into $k$ different sets such that for each set, there is no edge connecting two nodes. (See Fig.~\ref{fig:general_digraph}).
    \item $\NN_i$ is the set of in-neighbors of a node $i$ without including $i$, and $\NN_i^i$ is the set of in-neighbors of a node $i$ \linebreak including $i$.
    \item $\NN_i^p$ is the set of nodes with a path to node $i$.
    \item $\NN^m$ is the set of measured nodes.
\end{itemize}

\subsection{Model class}\label{sec:model_class}

We consider the identifiability problem in a network where the dynamics are additive on the edges and have finite memory. This particular setting is a special case of general nonlinear network models \cite{zanudo2017structure} and can be found in several applications like information distortion caused by transmission, sensors or privacy protection \cite{wang2022transmission}, nonlinear consensus algorithms including interval \cite{fontan2020interval} and  discarded consensus \cite{liu2012discarded}, and complex compartmental systems characterized by nonlinear flows. Moreover, this additive nonlinear model allows us to exploit the nonlinearity of the functions for the derivation of identifiability conditions, which is a key step in the extension of the results for more general nonlinear models. Finally, notice that the model is a nonlinear version of the linear model extensively analyzed in the identifiability of networks with full excitations \cite{weerts2018identifiability,hendrickx2019identifiability,vanwaarde2020necessary}.

The network is characterized by a weakly connected digraph $G=(V,E)$ composed by a set of nodes $V=\{ 1,\ldots,n\}$ and a set of edges $E\subseteq V\times V$. 
The output of each node $i$ in the network is given by: 
\begin{equation}\label{eq:nonlinear_model}
y_i^k=\sum_{j\in \mathcal{N}_i}f_{i,j}(y_j^{k-0},y_j^{k-1},y_j^{k-2},\ldots,y_j^{k-m_j})+u_i^{k-1}, 
\end{equation}
where $y_i^k$ is the state of the node $i$, $u_i^{k-1}$ is an external excitation signal, $\NN_i$ is the set of in-neighbors of node $i$ without including $i$, $f_{i,j}$ is a nonzero nonlinear function depending on past outputs of the neighbors $j$ of the node $i$, and each delay $m_j\in\mathbb{Z}^{\ge}$ is finite (i.e., finite memory).
The model \eqref{eq:nonlinear_model} is causal if the nonlinear functions $f_{i,j}$ do not depend on $y_j^{k-0}$ (i.e., $f_{i,j}(y_j^{k-1},y_j^{k-2},\ldots,y_j^{k-m_j})$ with $m_j\in\mathbb{Z}^+$),   and along this work we will only consider the identifiability of causal models (real-world dynamics). The particular case of a delay $m_j=0$ will be used only for the unfolded digraph introduced in Section~\ref{sec:general_digraphs}, which is a mathematical construction used in the proof of Theorem~\ref{thm:general_digraph}. The model \eqref{eq:nonlinear_model} implies that the output of a node $i$ in a network is determined by its own excitation signal $u_i$ and the past outputs of its in-neighbors $y_j$. The dynamics that determine the relationship between the node $i$ and an in-neighbor $j$ are given by a nonlinear function $f_{i,j}$ located in the corresponding edge.
Furthermore, the model \eqref{eq:nonlinear_model} corresponds to the full excitation case where all the nodes in the network can be excited through the inputs $u_i$. When the functions $f_{i,j}$ are linear, the identifiability problem with full excitations has been studied in \cite{hendrickx2019identifiability}. In this work, we will consider continuously and twice continuously differentiable functions, whose smoothness will allow us to derive conditions for the identifiability of nonlinear networks.

\begin{assumption}\label{ass:full_excitation}
    The graph $G$ associated with the network is known, where an edge $(i,j)\in E$ implies $f_{i,j}\not\equiv  0$.
\end{assumption}

Assumption~\ref{ass:full_excitation} implies that we know which nodes are connected by nonzero functions. By using a definition analogous to the adjacency matrix of a digraph, we consider that the absence of an edge $(i,j)\not\in E$ implies a zero function $f_{i,j}\equiv 0$, and the case $f_{i,j}\equiv 0$ is explicitly excluded if there is an edge $(i,j)\in E$. The identifiability analysis will be based on the notion of identifiability in system identification \cite{ljung1999system}, where the objective is to determine if there is a unique set of local dynamics $f_{i,j}$ that leads to a global behavior. Similarly to \cite{hendrickx2019identifiability,legat2020local,legat2021path}, we assume that in an ideal scenario the relations between excitations and outputs of the nodes have been perfectly identified (global behavior), and all nodes are excited. In the first part, we focus on networks that do not contain any cycle (i.e., DAGs). Due to the absence of cycles and finite memory, the function $F_i$ associated with the measurement of a node $i$, depends on a finite number of inputs
\begin{multline}\label{eq:function_Fi}
    y_i^k=u_i^{k-1}+F_i(u_1^{k-2},\ldots,u_1^{k-T_1},\ldots,u_{n_i}^{k-2},\ldots,u_{n_i}^{k-T_{n_i}}),\\
    1,\dots,n_i\in \mathcal{N}_i^p,
\end{multline}
where $\NN_i^p$ is the set of nodes with a path to node $i$. The function $F_i$ is implicitly defined by \eqref{eq:nonlinear_model} and the inputs in $F_i$ correspond to the nodes that have a path to the measured node~$i$. With a slight abuse of notation, we use the superscript in the function $F_i^{(s)}$ to indicate that all the inputs in \eqref{eq:function_Fi} are delayed by $s\in \mathbb{Z}^+$: $F_i^{(s)}=F_i(u_1^{k-2-s},\ldots,u_{n_i}^{k-T_{n_i}-s})$. 

\begin{figure}[!t]
    \centering
    \begin{tikzpicture}
    [
roundnodes/.style={circle, draw=black!60, fill=black!5, very thick, minimum size=1mm},roundnode/.style={circle, draw=white!60, fill=white!5, very thick, minimum size=1mm}
]
\node[roundnodes](node1){1};
\node[roundnodes](node2)[above=of node1,yshift=2mm,xshift=1.5cm]{2};
\node[roundnodes](node3)[right=2.5cm of node1]{3};
\draw[-{Classical TikZ Rightarrow[length=1.5mm]}] (node1) to node [left,swap,yshift=1mm] {$f_{2,1}$} (node2);
\draw[-{Classical TikZ Rightarrow[length=1.5mm]}] (node2) to node [right,swap,yshift=1mm] {$f_{3,2}$} (node3);
\draw[-{Classical TikZ Rightarrow[length=1.5mm]}] (node1) to node [below,swap] {$f_{3,1}$} (node3);

\node[roundnode](u1)[above=of node1,yshift=-12mm,xshift=-12mm]{$u_1$};
\node[roundnode](u2)[above=of node2,yshift=-12mm,xshift=-12mm]{$u_2$};
\node[roundnode](u3)[above=of node3,yshift=-12mm,xshift=12mm]{$u_3$};

\draw[gray,dashed,-{Classical TikZ Rightarrow[length=1.5mm]}] (u1) -- (node1);
\draw[gray,dashed,-{Classical TikZ Rightarrow[length=1.5mm]}] (u2) -- (node2);
\draw[gray,dashed,-{Classical TikZ Rightarrow[length=1.5mm]}] (u3) -- (node3);

\end{tikzpicture}
    \caption{DAG where the function associated with the measurement of the node $3$ only depends on a finite number of inputs.}
    \label{fig:DAG_volterra}
\end{figure}
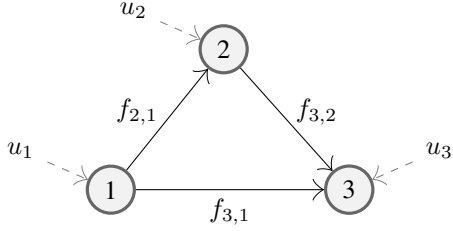

\begin{example}[Finite inputs]\label{ex:finite_delays}
    Let us consider the DAG in Fig.~\ref{fig:DAG_volterra} with nonlinear functions of the form $f_{3,2}(y_2^{k-1})$,  $f_{3,1}(y_1^{k-1},y_1^{k-2})$ and $f_{2,1}(y_1^{k-1})$. The measurement of the node 3 provides an output of the form
    \begin{align}
        y_3^k&=u_3^{k-1}+f_{3,2}(y_2^{k-1})+f_{3,1}(y_1^{k-1},y_1^{k-2})\nonumber\\
        &=u_3^{k-1}+f_{3,2}(u_2^{k-2}+f_{2,1}(u_1^{k-3}))+f_{3,1}(u_1^{k-2},u_1^{k-3})\label{eq:finite_delays},
    \end{align}
    so that the function $F_3$ associated with the measurement is a function of a finite number of variables $F_3(u_1^{k-2},u_1^{k-3},u_2^{k-2})$.
\end{example}

The model \eqref{eq:function_Fi} considered for the identifiability is of the type Nonlinear Finite Impulse Response (NFIR) where the output is a nonlinear function only of the excitation signals (i.e., inputs) \cite{ramirez2021nonlinear,pillonetto2025deep,nelles2020nonlinear}. When the functions $f_{i,j}$ are linear, the model \eqref{eq:nonlinear_model} corresponds to the affine dynamics that are extensively used in network systems \cite{ravazzi2014ergodic}. By considering that NFIR models can be accurately estimated by using different techniques like artificial neural networks \cite{ramirez2021nonlinear}, least squares \cite{makila2005lti}, etc., we make the following assumption.

\begin{assumption}\label{ass:knowledge_Fi}
If the node $i$ is measured, the function $F_i$ is known.
\end{assumption}

Assumption~\ref{ass:knowledge_Fi} implies a knowledge of the global behavior $F_i$ given by the relationships between the inputs and outputs. However, the knowledge of the function $F_i$ should not be taken in the context of an algorithm to actually compute the functions $f_{i,j}$. That is why errors during a potential identification process are not considered for the analysis. The objective of this work is
to determine which nodes need to be  measured
to identify all the nonlinear functions in the network. Our
aim is to determine the possibility of identification and not
to develop an algorithm or verify the accuracy of other
identification methods.

\subsection{Identifiability}

We first define the relationships between the measurements of the nodes and the functions $f_{i,j}$.

\begin{definition}[Set of measured functions]\label{def:set_measured_functions}
    Given a set of measured nodes $\mathcal{N}^m$, the totally ordered set of measured functions $(F(\mathcal{N}^m),\le)$ associated with $\mathcal{N}^m$ is given by:
    $$
    F(\mathcal{N}^m):=\{F_i\;|\;i\in \mathcal{N}^m\},
    $$
    with $F_i\le F_j$ if $i\le j$.
\end{definition}
Since the identifiability problem can be hard or unrealistic for general functions, we restrict the problem to  a certain class of functions $\F$: the functions associated with the edges belong to $\F$ and the identifiability is considered only among the functions belonging to $\F$. The different classes of functions will be defined depending on the different types of network topologies considered for the analysis.  

We say that a digraph $G$ and a partially ordered set of functions $(\{ f \},\le)=\{f_{i,j}\in \F\;|\;(i,j)\in E\}$ with $f_{i,j}\le f_{m,n}$ if $i\le m$ and $j\le n$, generate $F(\mathcal{N}^m)$ if the functions $F_i\in F(\mathcal{N}^m)$ are obtained \linebreak through \eqref{eq:nonlinear_model}. Notice that the ordering of the functions depends on the labeling of the nodes, which could be arbitrary and does not play a significant role. Therefore, in the rest of the paper, we will refer to $F(\mathcal{N}^m)$ and $\{ f \}$ only as a set of measured functions and a set of functions respectively.

\begin{definition}[Identifiability]\label{def:math_identifiability}
    Given a set of functions $\{ f \}=\{f_{i,j}\in \F\;|\;(i,j)\in E\}$ that generates $F(\NN^m)$ and any other set of functions $\{ \tilde f \}=\{\tilde f_{i,j}\in\F\;|\;(i,j)\in E\}$ that generates $\tilde F(\NN^m)$. An edge $f_{i,j}$ is identifiable in a class $\F$ if $
    F(\NN^m)=\tilde F(\NN^m)
    $ implies that $f_{i,j}=\tilde f_{i,j}$ for any $\tilde f_{i,j}$.
    A network $G$ is identifiable in a class $\F$ if 
    $
    F(\NN^m)=\tilde F(\NN^m)
    $
    implies that
    $
    \{ f \}=\{ \tilde f \}
    $ for any $\{ \tilde f \}$.
\end{definition}

According to Definition~\ref{def:math_identifiability}, in a network $G$, an edge $f_{i,j}$ is identifiable in a class $\F$ if given a set of measured functions $F(\mathcal{N}^m)$, every set of functions in $\F$ leading to $F(\mathcal{N}^m)$ has the same $f_{i,j}$ \cite{vizuete2023nonlinear}. A network $G$ is identifiable in a class $\F$ if all the edges are identifiable in the \linebreak class $\F$ \cite{vizuete2023nonlinear}.

Although Assumption~\ref{ass:knowledge_Fi} seems strong, it is clear that if a function $f_{i,j}$ cannot be identified with $F_i$, it is unidentifiable based on the measurement of this node. However, if the function $f_{i,j}$ is identifiable based on $F_i$, it is expected that a good approximation of $F_i$ with experiments, may eventually provide a good approximation of the function $f_{i,j}$. Thus, the possibility of identifying $f_{i,j}$ could be determinant for the approximation of these functions based on measurements and experiments \cite{ljung1999system}.

\section{Directed acyclic graphs (DAGs)}\label{sec:DAG}

We now begin by deriving identifiability conditions  for networks that do not contain any cycle (i.e., DAGs), considering specific classes of functions.

DAGs encompass a large number of graph topologies that present specific characteristics that can be used for the derivation of conditions for identifiability \cite{cheng2024identifiability,mapurunga2022excitation}. Since sources and sinks are important nodes in DAGs, we first provide a result similar to the static case in \cite{vizuete2023nonlinear} about the information that we can obtain from the measurements of these nodes. This result is not exclusively related to DAGs since sinks and sources can also exist in more general digraphs including cycles \footnote{Notice that the proof of Proposition~\ref{prop:sinks_sources} is valid for any $k$, and therefore, also holds for any digraph according to Definitions~\ref{def:new_measured_nodes} and \ref{def:new_math_identifiability} introduced in Section~\ref{sec:general_digraphs}.}.

\begin{proposition}[Sinks and sources]\label{prop:sinks_sources}
The measurement of the sources do not affect the identifiability of the network. For any sink $i$, its incoming edges are not identifiable if $i$ is not measured.

\end{proposition}
\begin{proof}
First, the measurement of any source $\ell$ generates the output $y_\ell^k=u_\ell^{k-1}$, which does not provide any information about functions associated with edges in the network. Then, let us consider a sink $i$ with $\aNN$ in-neighbors. The measurement of this sink provides the output
\begin{align*}
y_i^k&=u_i^{k-1}+F_i\\
&=u_i^{k-1}+\sum_{j\in\NN_i}f_{i,j}(y_j^{k-1},\ldots,y_j^{k-m_j}).  
\end{align*}
If the sink $i$ is not measured, the function $F_i\not\in F(\NN^m)$. For any set $\{ f \}$ that generates $F(\NN^m)$, we can obtain another set $\{ \tilde f \}$ where all the edges are the same except the functions $f_{i,j}$, that will also generate $F(\NN^m)$, since these functions $f_{i,j}$ do not affect any of the functions in $F(\NN^m)$. Thus, if the sink $i$ is not measured, the incoming edges $f_{i,j}$ are not identifiable.
\end{proof}

\subsection{Identifiability of general nonlinear functions}

We start by considering the identifiability problem in the class of all functions.

\begin{definition}[Class of functions $\Fzero$]
Let $\Fzero$ be the class of all  functions $f:\R^m\to\R$.
\end{definition}

Notice that the dimension $m$ corresponds to the function $f_{i,j}$ in the network with the largest number of arguments (i.e., number of past inputs).

Unfortunately, the following proposition shows that even in the trivial case when we measure all the nodes, there are some important network topologies where the identifiability problem is unsolvable in any class invariant by adding constants, which includes the class $\Fzero$. 

\begin{proposition}[Unidentifiability]\label{prop:DAG_unidentifiable}
If a DAG contains two (or more) edges arriving at a same node, then none of them are identifiable in any class invariant by adding constants. 
\end{proposition}
\begin{proof}
     Let us consider an arbitrary DAG with a node $i$ with 2 or more in-neighbors. Let us assume that we measure all the nodes (i.e., $\NN^m=V$), which is all the information that we can get from a network. Since the function $F_i\in F(\NN^m)$, we obtain:
 \begin{align*}
    y_i^k&=u_i^{k-1}+F_i\\
    &=u_i^{k-1}\!+\!\sum_{j\in \NN_i} \! f_{i,j}(u_j^{k-2}\!+\!F_j^{(1)},\ldots,u_j^{k-m_j-1}\!+\!F_j^{(m_j)}).  
 \end{align*}
 Let us consider two arbitrary in-neighbors $p$ and $q$ of the node $i$ and set the functions $\tilde f_{i,p}=f_{i,p}+\gamma$ and  $\tilde f_{i,q}=f_{i,q}-\gamma$, with $\gamma\neq 0$. Notice that $F_i$ can also be obtained with the functions  $\tilde f_{i,p}$ and $\tilde f_{i,q}$, which implies that these edges cannot be identified.
\end{proof}

\begin{figure}
    \centering
    \begin{tikzpicture}
    [
roundnodes/.style={circle, draw=black!60, fill=black!5, very thick, minimum size=1mm},
roundnode/.style={circle, dashed, draw=black!60, fill=white!5, very thick, minimum size=3cm},
punto/.style={circle, dashed, draw=white!60, fill=white!5, very thick, minimum size=1mm}
]

\node[roundnode](node1){};
\node[roundnodes](node2)[above=of node1,yshift=-1.5cm,xshift=3cm]{$p$};
\node[roundnodes](node4)[right=3cm of node1]{$i$};
\node[roundnodes](node3)[below=of node1,yshift=1.5cm,xshift=3cm]{$q$};
\node[punto](node5)[left=2cm of node2]{};
\node[punto](node6)[left=2cm of node3]{};

\draw[-{Classical TikZ Rightarrow[length=1.5mm]}] (node5) to node [left,swap,yshift=2.5mm] {} (node2);
\draw[-{Classical TikZ Rightarrow[length=1.5mm]}] (node2) to node [left,swap,yshift=0mm,xshift=-1mm] {$f_{i,p}$} (node4);
\draw[-{Classical TikZ Rightarrow[length=1.5mm]}] (node2) to node [right,swap,yshift=0mm,xshift=2mm] { $\tilde f_{i,p}=f_{i,p}+\gamma$} (node4);
\draw[-{Classical TikZ Rightarrow[length=1.5mm]}] (node6) to node [left,swap,yshift=-2.5mm] {} (node3);
\draw[-{Classical TikZ Rightarrow[length=1.5mm]}] (node3) to node [right,swap,yshift=-0mm,xshift=2mm] { $\tilde f_{i,q}=f_{i,q}-\gamma$} (node4);

\draw[-{Classical TikZ Rightarrow[length=1.5mm]}] (node3) to node [left,swap,yshift=-0mm,xshift=-1mm] {$f_{i,q}$} (node4);

\node[roundnode](node7){\Large DAG};
\end{tikzpicture}
    \caption{If a DAG has at least a node $i$ with two or more in-neighbors, then it is not identifiable in the class $\Fzero$.}
    \label{fig:unidentified_network}
\end{figure}
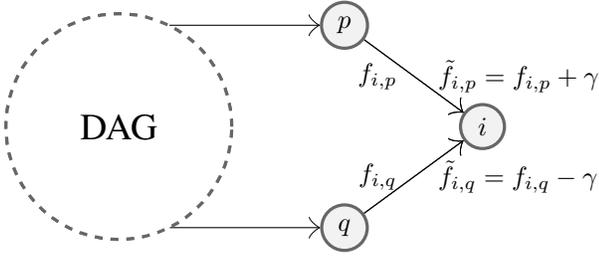

Proposition~\ref{prop:DAG_unidentifiable} implies that DAGs having nodes with two or more in-neighbors cannot be identified (see for instance Fig.~\ref{fig:unidentified_network}), which restricts considerably the types of networks that can be identified in the class $\Fzero$. In the following proposition we show that the identifiability of graphs where all the nodes have only 1 in-neighbor (i.e., an arborescence) is possible but requires the measurement of all the nodes except the source. 

\begin{proposition}[Arborescence]\label{prop:path_general_suff}
    An arborescence is identifiable in the class $\Fzero$ if and only if all the nodes are measured except possibly the source.
\end{proposition}
\begin{proof}
    Let us consider an arborescence with $n>2$ nodes and a node $i$, which is neither the source nor a sink \footnote{An arborescence can have several sinks but only one source.}. The output of an out-neighbor of $i$ denoted by $i+1$ is given by
\begin{align}
y_{i+1}^k&=u_{i+1}^{k-1}+F_{i+1}\label{eq:F_ik}\\
&=u_{i+1}^{k-1}+f_{i+1,i}(u_i^{k-2}+F_i^{(1)},\ldots,u_i^{k-m_i-1}+F_i^{(m_i)})\nonumber\\
&=u_{i+1}^{k-1}+f_{i+1,i}(u_i^{k-2}+f_{i,i-1}(u_{i-1}^{k-3}+F_{i-1}^{(2)},\ldots),\ldots, \nonumber\\
&\quad\:\: u_i^{k-m_i-1}+f_{i,i-1}(u_{i-1}^{k-m_i-2}+F_{i-1}^{(m_i-1)},\ldots)),\nonumber
\end{align}
where $i-1$ is the in-neighbor of $i$.
If the node $i$ is not measured, we can consider the functions $\tilde f_{i,i-1}(\x)=f_{i,i-1}(\x)+\gamma$ and $\tilde f_{i+1,i}(\x)=f_{i+1,i}(x_1-\gamma,\ldots,x_{m_i}-\gamma)$ with $\gamma\neq 0$, that generate $\tilde F_{i+1}$ and satisfy $F_{i+1}=\tilde F_{i+1}$ in \eqref{eq:F_ik}, which implies that the arborescence cannot be identified. 

\noindent On the other hand, if we measure all the nodes (i.e., $\NN^m=V$), we know the function $F_j$ associated with any node $j$ in the network. Let us consider a node $i$ with an in-neighbor and we set all the inputs to zero except $u_i$ with the corresponding delays. Then, the measurement of the node~$i$ gives us:
\begin{align}
    y_i^k&=u_i^{k-1}+F_i\label{eq:F_iik}\\
    &=u_i^{k-1}+f_{i,i-1}(u_i^{k-2}+F_{i-1}^{(1)}(\zero),\ldots,\nonumber\\
    &\quad\:\: u_i^{k-m_{i-1}-1}+F_{i-1}^{(m_{i-1})}(\zero)).\nonumber
\end{align}
If there is another function $\tilde f_{i,i-1}$ that generates $\tilde F_i$ and satisfies $F_i=\tilde F_i$ in \eqref{eq:F_iik}, we would have for all $u_i^{k-2},\ldots,u_i^{k-m_{i-1}-1}\in\R$:
\begin{multline*}
f_{i,i-1}(u_i^{k-2}+F_{i-1}(\zero),\ldots,u_i^{k-m_{i-1}-1}+F_{i-1}^{(m_{i-1})}(\zero))=\\
\tilde f_{i,i-1}(u_i^{k-2}+F_{i-1}(\zero),\ldots,u_i^{k-m_{i-1}-1}+F_{i-1}^{(m_{i-1})}(\zero)),    
\end{multline*}
which implies that $f_{i,i-1}=\tilde f_{i,i-1}$ because all the $F_{i-1}$ are the same, and we thus identify $f_{i,i-1}$. Since the node $i$ was arbitrary, the analysis holds for the other nodes and therefore all the nonlinear functions in the arborescence are identified.
By Proposition~\ref{prop:sinks_sources}, it is never necessary to measure the source. Finally, since each edge between two measured nodes is identified, it is sufficient to measure all the nodes except possibly the source, which does not provide information.
\end{proof}

We can see that even for simple topologies like arborescences, the identifiability problem requires the measurement of all the nodes except the source, which could be impractical due to physical limitations and economic costs.
Furthermore, as Proposition~\ref{prop:DAG_unidentifiable} showed, for more complex graph topologies, even the measurement of all the nodes is not enough to identify all the nonlinear functions in a network. From the proof of Proposition~\ref{prop:DAG_unidentifiable}, we can see that the main problem is the identification of the static part of the nonlinear functions since it requires more information through the measurement of additional nodes.
Therefore, we will consider nonlinear functions where no static part is present (i.e., only a dynamic part), which could be interpreted as the difference with respect to the steady state. In addition, we will consider continuously differentiable functions \footnote{Unfortunately, we need to discard almost everywhere continuously differentiable functions (e.g., ReLU) since the composition of these functions does not necessarily belong to the same class.}, which are smooth enough to guarantee the identifiability of path graphs and trees. 

\begin{definition}[Class of functions $\Fone$]
Let $\Fone$ be the class of functions $f:\R^m\to\R$ with the following properties:
\begin{enumerate}
    \item $f$ is continuously differentiable in $\R^m$.
    \item $f(\zero)=0$.
\end{enumerate}
\end{definition}
The class $\Fone$ covers numerous nonlinear
functions \cite{abramowitz1965handbook}, including polynomial functions which are used for the approximation of continuous functions through the Weierstrass Approximation theorem \cite{llavona1986approximation}.
    
\subsection{Path graphs and trees}\label{sec:paths_trees}

The notion of paths plays an important role in the identifiability problem in networks since they carry information between the excited nodes and the measured nodes. For this reason, in this section, we focus on the identifiability of path graphs. First, we need the following lemma about the identifiability of in-neighbors, which also holds in the linear case (see Theorem~V.1 in \cite{hendrickx2019identifiability}).

\begin{lemma}\label{lemma:sinks}
Let $G$ be a DAG such that the edge $(i,j)$ is the only path from the node $j$ to the node $i$. Then if $i$ is measured, $f_{i,j}$ is identifiable in the class $\Fone$.
\end{lemma}
\begin{proof}
The measurement of the node $i$ provides the output:
\begin{equation}\label{eq:F_incoming_edges}
 y_i^k=u_i^{k-1}+ \sum_{\ell\in\NN_i}f_{i,\ell}(u_{\ell}^{k-2}+ F_{\ell}^{(1)},\ldots,u_{\ell}^{k-m_{\ell}-1}+ F_{\ell}^{(m_{\ell})}).
\end{equation}
Since $F_i\in F(\NN^m)$, let us assume that there exists a set $\{\tilde f\}\neq \{ f\}$ such that $F_i=\tilde F_i$, which implies:
\begin{multline}\label{eq:inter_one_edge}
    \sum_{\ell\in\NN_i}f_{i,\ell}(u_{\ell}^{k-2}+ F_{\ell}^{(1)},\ldots,u_{\ell}^{k-m_{\ell}-1}+ F_{\ell}^{(m_{\ell})})=\\\sum_{\ell\in\NN_i}\tilde f_{i,\ell}(u_{\ell}^{k-2}+ \tilde F_{\ell}^{(1)},\ldots,u_{\ell}^{k-m_{\ell}-1}+ \tilde F_{\ell}^{(m_{\ell})}).
\end{multline}
Let us set to zero all the inputs except the inputs of the node $j\in\NN_i$ with their respective delays. Since $(i,j)$ is the only path between $i$ and $j$, the functions $F_\ell$, $\tilde F_\ell$ do not depend on the inputs $u_j$, and given that all the functions are in the class $\Fone$, the equality \eqref{eq:inter_one_edge} becomes:
\begin{multline*}
    f_{i,j}(u_{j}^{k-2}+ F_{j}^{(1)}(\zero),\ldots,u_{j}^{k-m_{j}-1}+ F_{j}^{(m_{j})}(\zero))=\\ \tilde f_{i,j}(u_{j}^{k-2}+ \tilde F_{j}^{(1)}(\zero),\ldots,u_{j}^{k-m_{j}-1}+ \tilde F_{j}^{(m_{j})}(\zero)).
\end{multline*}
which yields for all $u_j^{k-2},\ldots,u_j^{k-m_j-1}\in\R$:
$$
f_{i,j}(u_j^{k-2},\ldots,u_j^{k-m_j-1})=\tilde f_{i,j}(u_j^{k-2},\ldots,u_j^{k-m_j-1}). 
$$
Therefore, we obtain $f_{i,j}=\tilde f_{i,j}$, which implies that the function $f_{i,j}$ is identifiable.
\end{proof}

This leads to the following corollary.

\begin{corollary}\label{corr:sinks}
    For identifiability in the class $\Fone$, the measurement of a node $i$ provides the identifiability of all the incoming edges corresponding to in-neighbors with only one path to $i$.
\end{corollary}

Due to the presence of past inputs, Corollary~\ref{corr:sinks} is different from Lemma 2 in our preliminary work \cite{vizuete2023nonlinear} corresponding to the static case, where the measurement of a node provides the identification of all the incoming edges of the node. For instance, let us consider again the network in Fig.~\ref{fig:DAG_volterra} and the nonlinear functions of Example~\ref{ex:finite_delays}. Notice that 
based on \eqref{eq:finite_delays}, the function $f_{3,2}$ can be identified by setting to zero $u_1^{k-2}$ and $u_1^{k-3}$, but the function $f_{3,1}$ cannot be identified by using the same approach since $u_1^{k-3}$ is also present in $f_{3,2}$.
This shows that the identifiability problem in the dynamic case is clearly more complex than in the static case.

Next, we provide the following technical lemma that will be used in the proofs of the main results, where we remind that $\tau$ is the shift operator of coordinates defined in Section~\ref{sec:notation}.

\begin{lemma}[Multivariate function]\label{lemma:multivariate}
Given three non identically zero continuously differentiable functions $f:\R^p\to \R$ and $g,\tilde g:\R^q\to\R$ that satisfy:
$$
g(\zero)=\tilde g(\zero)=0;
$$
\vspace{-7mm}
\begin{multline}\label{eq:equality_multivariate}
f(x_1\!+\!g(\y_1,\ldots,\y_r),\ldots,x_p\!+\!g(\tau^{p-1}\y_1,\ldots,\tau^{p-1}\y_r))=\\     
f(x_1\!+\!\tilde g(\y_1,\ldots,\y_r),\ldots,x_p\!+\!\tilde g(\tau^{p-1}\y_1,\ldots,\tau^{p-1}\y_r)),
\end{multline}
for all $x_i$ with $i\in \{1,\ldots,p\}$, and for all $\y_j$ with $j\in \{1,\ldots,r\}$,
then either $g=\tilde g$ or $f$ is constant.
\end{lemma}
\begin{proof}
    The proof is deferred to Appendix~\ref{app:1}
\end{proof}

Since the functions in the class $\Fone$ satisfy $f(0)=0$, and the function $f$ in Lemma~\ref{lemma:multivariate} cannot be identically zero, we have the following corollary.

\begin{corollary}\label{corr:multivariate}
    Under the same conditions as in Lemma~\ref{lemma:multivariate}, if $f(\zero)=0$, then
    $$
    g=\tilde g.
    $$
\end{corollary}

\begin{proposition}[Path graphs]\label{prop:paths_identifiability}
A path graph is identifiable in the class $\Fone$, if and only if the sink is measured.
\end{proposition}
\begin{proof}
    By Proposition~\ref{prop:sinks_sources} the measurement of the sink $i$ is necessary. The measurement gives us the output:
    \begin{align*}
        y_i^k&\!=\!u_i^{k-1}+F_i\\
        &\!=\!u_i^{k-1}\!+\!f_{i,i-1}(u_{i-1}^{k-2}\!+\!F_{i-1}^{(1)},\ldots,u_{i-1}^{k-m_{i-1}-1}\!+\!F_{i-1}^{(m_{i-1})}).
    \end{align*}
    Since $F_i\in F(\NN^m)$, let us assume that there exists a set $\{\tilde f\}\neq \{ f\}$ such that $F_i=\tilde F_i$, which corresponds to the equality:
    \begin{multline*}
        f_{i,i-1}(u_{i-1}^{k-2}+F_{i-1}^{(1)},\ldots,u_{i-1}^{k-m_{i-1}-1}+F_{i-1}^{(m_{i-1})})=\\
        \tilde f_{i,i-1}(u_{i-1}^{k-2}+\tilde F_{i-1}^{(1)},\ldots,u_{i-1}^{k-m_{i-1}-1}+\tilde F_{i-1}^{(m_{i-1})}).
    \end{multline*}
    Notice that $F_{i-1}$ and $\tilde F_{i-1}$ could at this stage be different since the node $i-1$ has not been measured.
    Between the nodes $i$ and $i-1$ there is only one path corresponding to the edge $(i,i-1)$, which implies that the function $f_{i,i-1}$ can be identified according to Lemma~\ref{lemma:sinks}, so that $f_{i,i-1}=\tilde f_{i,i-1}$. Then, we get:
    \begin{multline*}
        f_{i,i-1}(u_{i-1}^{k-2}+F_{i-1}^{(1)},\ldots,u_{i-1}^{k-m_{i-1}-1}+F_{i-1}^{(m_{i-1})})=\\
         f_{i,i-1}(u_{i-1}^{k-2}+\tilde F_{i-1}^{(1)},\ldots,u_{i-1}^{k-m_{i-1}-1}+\tilde F_{i-1}^{(m_{i-1})}).
    \end{multline*}
    By virtue of Lemma~\ref{lemma:multivariate} we can guarantee $F_{i-1}^{(1)}=\tilde F_{i-1}^{(1)}$. This is equivalent to the mathematical constraint obtained with the measurement of the node $i-1$
    (i.e., $F_{i-1}\in F(\NN^m$)) and we can continue with the same procedure until we reach the root. In this way, we guarantee $\{ f \}=\{ \tilde f \}$ and all the path can be identified.
\end{proof}

\begin{proposition}[Trees]\label{prop:trees_identifiability}
A tree is identifiable in the class $\Fone$, if and only if all the sinks are measured.
\end{proposition}
\begin{proof}
    By Proposition~\ref{prop:sinks_sources} the measurement of all the sinks is necessary.
    Let us consider one of the sinks $i$ of an arbitrary tree. The output is given by:
    \begin{align}
        y_i^k&=u_i^{k-1}+F_i\nonumber\\
        &=u_i^{k-1}\!+\!\sum_{j\in\NN_i}\!f_{i,j}(u_{j}^{k-2}\!+\!F_{j}^{(1)},\ldots,u_{j}^{k-m_{j}-1}\!+\!F_{j}^{(m_{j})}).\label{eq:branches_tree}
    \end{align}
    Given that $F_i\in F(\NN^m)$, let us assume that there exists a set $\{\tilde f\}\neq \{ f\}$ such that $F_i=\tilde F_i$, which implies:
    \begin{multline}\label{eq:equal_f_tilde}
        \sum_{j\in\NN_i}f_{i,j}(u_{j}^{k-2}+F_{j}^{(1)},\ldots,u_{j}^{k-m_{j}-1}+F_{j}^{(m_{j})})=\\
        \sum_{j\in\NN_i}\tilde f_{i,j}(u_{j}^{k-2}+\tilde F_{j}^{(1)},\ldots,u_{j}^{k-m_{j}-1}+\tilde F_{j}^{(m_{j})}).
    \end{multline}
    Since in a tree, between two nodes, there is only one path, we can choose an arbitrary path that finishes in the sink $i$ and set to zero the excitation signals of all the nodes that do not belong to that path. In this way, the equality \eqref{eq:equal_f_tilde} corresponds to the identifiability of a path graph and by Proposition~\ref{prop:paths_identifiability} all the edges can be identifiable. We can use the same procedure for other paths that finish in the sink $i$ to identify all the branches that end in $i$, and by measuring the other sinks, we can identify all the edges in the tree.
\end{proof}

We recall that the case of arborescences are also included in Proposition~\ref{prop:trees_identifiability}. The identifiability conditions of Propositions~\ref{prop:paths_identifiability} and \ref{prop:trees_identifiability} are similar to the linear case where the measurement of all the sinks is also necessary and sufficient for identifiability of trees (see Theorem~IV.3 in \cite{hendrickx2019identifiability}). 
Notice that the class $\Fone$ also includes linear functions, which implies that the conditions must be at least as strong as in the linear case, but not necessarily the same.
Furthermore, Propositions~\ref{prop:paths_identifiability} and \ref{prop:trees_identifiability} show that the nonlinearity of the dynamics does not require the measurement of additional nodes. In the next subsection, we will see that the identifiability conditions can, in fact, be relaxed for more general DAGs thanks to the nonlinearities.

\subsection{General DAGs}\label{sec:DAGs}

Unlike the particular case of trees, in more general DAGs, the functions $F_j^{(1)}$ in \eqref{eq:branches_tree} could depend on common variables due to several possible paths between two nodes with the same length, which makes impossible to isolate parts of the network through appropriate zero inputs. In addition, some functions in the class $\Fone$ might not be identifiable in a DAG.

\begin{example}
Let us consider the network in Fig.~\ref{fig:bridge_graph} with nonlinear functions of the form $f_{4,3}(y_3^{k-1})$, $f_{4,2}(y_2^{k-1})$, $f_{2,1}(y_1^{k-1})$ and $f_{3,1}(y_1^{k-1})$. The measurement of the node 4 provides an output of the form

\vspace{-3mm}

\small
\begin{align*}
    y_4^k&\!=\!u_4^{k-1}+f_{4,2}(y_2^{k-1})+f_{4,3}(y_3^{k-1})\\
    &\!=\!u_4^{k-1}\!+\!f_{4,2}(u_2^{k-2}\!+\!f_{2,1}(u_1^{k-3}))\!+\!f_{4,3}(u_3^{k-2}\!+\!f_{3,1}(u_1^{k-3})).
\end{align*}

\normalsize
\noindent The excitation signal $u_1^{k-3}$ appears in the functions $f_{4,2}$ and $f_{4,3}$, so that the different paths from 1 to 4 cannot be isolated by setting to zero the corresponding excitation signals. Furthermore, let us assume that the functions $f_{4,2}=a_1u_2^{k-1}$ and $f_{4,3}=a_2u_3^{k-1}$ are linear. Then, we have:

\vspace{-3mm}

\small
\begin{equation}\label{eq:linear_functions}
    y_4^k\!=\!u_4^{k-1}\!+\!a_1u_2^{k-2}\!+\!a_2u_3^{k-2}\!+\!a_1f_{2,1}(u_1^{k-3})\!+\!a_2f_{3,1}(u_1^{k-3}).
\end{equation}

\normalsize

\noindent Notice that the functions $\hat f_{2,1}=\frac{a_2}{a_1}f_{3,1}$ and $\hat f_{3,1}=\frac{a_1}{a_2}f_{2,1}$ also satisfy \eqref{eq:linear_functions}. This is due to the superposition principle that does not allow to distinguish the information of node 1 that arrives through the node 2 or 3.    
\end{example}

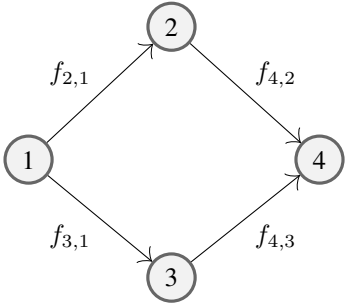
\begin{figure}
    \centering
    \begin{tikzpicture}
    [
roundnodes/.style={circle, draw=black!60, fill=black!5, very thick, minimum size=1mm},roundnode/.style={circle, draw=white!60, fill=white!5, very thick, minimum size=1mm}
]
\node[roundnodes](node1){1};
\node[roundnodes](node2)[above=of node1,yshift=1mm,xshift=1.9cm]{2};
\node[roundnodes](node4)[right=3.2cm of node1]{4};
\node[roundnodes](node3)[below=of node1,yshift=1mm,xshift=1.9cm]{3};

\draw[-{Classical TikZ Rightarrow[length=1.5mm]}] (node1) to node [left,swap,yshift=2.5mm] {$f_{2,1}$} (node2);
\draw[-{Classical TikZ Rightarrow[length=1.5mm]}] (node2) to node [right,swap,yshift=2.5mm] {$f_{4,2}$} (node4);
\draw[-{Classical TikZ Rightarrow[length=1.5mm]}] (node1) to node [left,swap,yshift=-2.5mm] {$f_{3,1}$} (node3);
\draw[-{Classical TikZ Rightarrow[length=1.5mm]}] (node3) to node [right,swap,yshift=-2.5mm] {$f_{4,3}$} (node4);
\end{tikzpicture}
    \caption{Network with functions in $\Fone$ that cannot be identified by only measuring the sink when $f_{2,1}$ and $f_{3,1}$ are linear due to the superposition principle.}
    \label{fig:bridge_graph}
\end{figure}

For this reason, we restrict the identifiability problem to a smaller class of functions \cite{vizuete2023nonlinear}.

\begin{definition}[Class of functions $\Ftwo$]\label{def:class_Ftwo}
Let $\Ftwo$ be the class of functions $f:\R^m\to\R$ with the following properties:
\begin{enumerate}
    \item $f$ is twice continuously differentiable in $\R^m$.
    \item $f(\zero)=0$.
    \item At least one of the partial derivatives of $f$ is not constant.    
\end{enumerate}
\end{definition}

Unlike the class $\Fone$, the last condition in Definition~\ref{def:class_Ftwo} is equivalent to none of the functions $f$ in $\Ftwo$ being purely linear, which still encompasses numerous nonlinear functions. Furthermore, we restrict the class to \emph{twice} continuously differentiable functions since the derivative will allow us to distinguish the information coming through different paths.
Since in a DAG, an input can influence a node via multiple paths, we cannot apply the argument of the proof of Proposition~\ref{prop:trees_identifiability} based on setting to zero specific inputs. 
For this reason, we need to exploit further properties of DAGs related to the number of paths between two nodes, to derive identifiability conditions.

\begin{lemma}\label{lemma:unique_path}
    In a DAG, any node $i$ which is not a source has at least one in-neighbor $j$ with only one path to $i$.   
\end{lemma}
\begin{proof}
If all the nodes except the sources have only one in-neighbor, the result trivially holds. When there is a node with more than one in-neighbor, we prove it by contradiction. Let us assume there is a DAG with a node $i$ with $\aNN\ge 2$ in-neighbors.  Let us suppose that all the in-neighbors of $i$ have more than one path to $i$ and let us consider the subgraph induced by the set of in-neighbors of $i$ and the nodes that correspond to any path between two in-neighbors of $i$. This subgraph does not have a sink, which contradicts the fact that any subgraph of a DAG is also a DAG. Therefore, there must be at least one neighbor $j$ with only one path to $i$, which would correspond to the sink of the induced subgraph. 
\end{proof}

For our convenience, we introduce the notion of  identifiability of a function $F_\ell$ that will be used in the proof of Lemma~\ref{lemma:removal_node}.

\begin{definition}[Identifiability of $F_\ell$]\label{def:math_identifiability_Fi}
    Given a set of functions $\{ f \}=\{f_{i,j}\in \F\;|\;(i,j)\in E\}$ that generates $F(\NN^m)$ and any other set of functions $\{ \tilde f \}=\{\tilde f_{i,j}\in\F\;|\;(i,j)\in E\}$ that generates $\tilde F(\NN^m)$. The function $F_\ell$ is identifiable if $
    F(\NN^m)=\tilde F(\NN^m)
    $, implies that $F_{\ell}=\tilde F_{\ell}$ for any $\tilde F_{\ell}$, which is equivalent to the mathematical constraint associated with the measurement of the node $\ell$.
\end{definition}

If the node $\ell$ is measured (i.e., $\ell\in\NN^m$), the function $F_\ell$ is trivially identifiable since the measurement of $\ell$ provides directly the knowledge of $F_\ell$.

The following lemma provides a relationship between the identifiability problem based on the measurement of a node $i$ in a graph $G$ and the identifiability problem based on the measurement of the same node $i$ in the induced subgraph $G_{V\setminus\{ j\}}$ obtained with the removal of an in-neighbor $j$.

\begin{lemma}[Removal of a node]\label{lemma:removal_node}
    Given a DAG $G$ and a node $j$ with only one path to its out-neighbor $i$. Let us assume that $j$ has been measured and the edge $f_{i,j}$ is identifiable. When $i\in \NN^m$, the edges $f_{i,\ell}$ and functions $F_\ell$ for $\ell\in\NN_i$ are identifiable in $G$ if they are identifiable in the induced subgraph   
    $G_{V\setminus\{ j\}}$.
\end{lemma}
\begin{proof}
In the DAG $G$, the measurement of the node $i$ provides the output:
\begin{align}
  y_i^k&=u_i^{k-1}+F_i\nonumber\\
  &=u_i^{k-1}+ \sum_{\ell\in\NN_i}f_{i,\ell}(u_{\ell}^{k-2}+ F_{\ell}^{(1)},\ldots,u_{\ell}^{k-m_{\ell}-1}+ F_{\ell}^{(m_{\ell})}),   \label{eq:measure_one_path}
\end{align}
    where the functions $F_\ell$ for $\ell\in\NN_i\setminus \{ j \}$ do not depend on any information coming through $j$ since there is no path from $j$ to other in-neighbors of $i$.
   Let us assume that there exists a set $\{\tilde f\}\neq \{ f\}$ such that $F_i=\tilde F_i$, which implies:
    \begin{multline*}
        \sum_{\ell\in\NN_i}f_{i,\ell}(u_{\ell}^{k-2}+F_{\ell}^{(1)},\ldots,u_{\ell}^{k-m_{\ell}-1}+F_{\ell}^{(m_{\ell})})=\\
        \sum_{\ell\in\NN_i}\tilde f_{i,\ell}(u_{\ell}^{k-2}+\tilde F_{\ell}^{(1)},\ldots,u_{\ell}^{k-m_{\ell}-1}+\tilde F_{\ell}^{(m_{\ell})}).
    \end{multline*}
    Since the measurement of $j$ implies the knowledge of $F_j$ (i.e., $F_j=\tilde F_j$), and the edge $f_{i,j}$ is identifiable (i.e., $f_{i,j}=\tilde f_{i,j}$), we get:   \begin{multline}\label{eq:equal_f_tilde_one_path}
        \sum_{\ell\in\NN_i\setminus\{ j\}}f_{i,\ell}(u_{\ell}^{k-2}+F_{\ell}^{(1)},\ldots,u_{\ell}^{k-m_{\ell}-1}+F_{\ell}^{(m_{\ell})})=\\
        \sum_{\ell\in\NN_i\setminus\{ j \}}\tilde f_{i,\ell}(u_{\ell}^{k-2}+\tilde F_{\ell}^{(1)},\ldots,u_{\ell}^{k-m_{\ell}-1}+\tilde F_{\ell}^{(m_{\ell})}).
    \end{multline}
    Now, let us consider the induced subgraph $G_{V\setminus\{ j \}}$. With the removal of the node $j$, the measurement of the node $i$ provides the output:
    \begin{equation}\label{eq:node_i_subgraph}
    y_i^k=u_i^{k-1}+ \sum_{\ell\in\NN_i\setminus\{ j\}}f_{i,\ell}(u_{\ell}^{k-2}+\hat F_{\ell}^{(1)},\ldots,u_{\ell}^{k-m_{\ell}-1}+\hat F_{\ell}^{(m_{\ell})}).    
    \end{equation}
    Since $j$ does not affect the functions $F_\ell$ for $\ell\in\NN_i\setminus \{ j \}$, we have that $\hat F_\ell=F_\ell$, so that \eqref{eq:node_i_subgraph} implies \eqref{eq:equal_f_tilde_one_path}. Therefore, if the edges $f_{i,\ell}$ and the functions $F_\ell$ for $\ell\in\NN_i\setminus\{j\}$ are identifiable in $G_{V\setminus\{ j \}}$, they are also identifiable in $G$. 
\end{proof}

\begin{figure}
    \centering
    \begin{tikzpicture}
    [
roundnodes/.style={circle, draw=black!60, fill=black!5, very thick, minimum size=1mm},roundnode/.style={circle, draw=white!60, fill=white!5, very thick, minimum size=1mm}
]
\node[roundnodes](node1){1};
\node[roundnodes](node2)[below=of node1,yshift=4mm,xshift=1cm]{2};
\node[roundnodes](node3)[below=of node2,yshift=4mm,xshift=-1cm]{3};
\node[roundnodes](node4)[below=of node3,yshift=4mm,xshift=-0.5cm]{4};
\node[roundnodes](node5)[right=of node2,yshift=0cm,xshift=0cm]{5};

\node[roundnodes](node6)[right=of node1,yshift=0mm,xshift=3cm]{1};
\node[roundnodes](node7)[right=of node3,yshift=0cm,xshift=3cm]{3};
\node[roundnodes](node8)[right=of node4,yshift=0mm,xshift=3cm]{4};
\node[roundnodes](node9)[right=of node5,yshift=0cm,xshift=3cm]{5};

\draw[-{Classical TikZ Rightarrow[length=1.5mm]}] (node1) to node [left,swap,yshift=2.5mm] {} (node2);
\draw[-{Classical TikZ Rightarrow[length=1.5mm]}] (node3) to node [right,swap,yshift=2.5mm] {} (node2);
\draw[-{Classical TikZ Rightarrow[length=1.5mm]}] (node4) to node [left,swap,yshift=-2.5mm] {} (node3);
\draw[-{Classical TikZ Rightarrow[length=1.5mm]}] (node1) to node [right,swap,yshift=-2.5mm] {} (node5);
\draw[-{Classical TikZ Rightarrow[length=1.5mm]}] (node2) to node [right,swap,yshift=-2.5mm] {} (node5);
\draw[-{Classical TikZ Rightarrow[length=1.5mm]}] (node3) to node [right,swap,yshift=-2.5mm] {} (node5);
\draw[-{Classical TikZ Rightarrow[length=1.5mm]}] (node4) to node [right,swap,yshift=-2.5mm] {} (node5);

\draw[-{Classical TikZ Rightarrow[length=1.5mm]}] (node6) to node [left,swap,yshift=2.5mm] {} (node9);
\draw[-{Classical TikZ Rightarrow[length=1.5mm]}] (node7) to node [left,swap,yshift=2.5mm] {} (node9);
\draw[-{Classical TikZ Rightarrow[length=1.5mm]}] (node8) to node [left,swap,yshift=2.5mm] {} (node9);
\draw[-{Classical TikZ Rightarrow[length=1.5mm]}] (node8) to node [left,swap,yshift=2.5mm] {} (node7);

\node(graph1)[below=of node4,yshift=0.7cm,xshift=1.6cm]{$a) $ Digraph $G$};

\node(graph1)[below=of node8,yshift=0.7cm,xshift=1.6cm]{$b) $ Induced subgraph $G_{V\setminus \{ 2\}}$};

\end{tikzpicture}
    \caption{If the node 2 is measured and $f_{5,2}$ is identifiable, the measurement of the node 5 provides the identifiability of the edges $f_{5,1}$, $f_{5,3}$ and $f_{5,4}$ and functions $F_1$, $F_3$ and $F_4$ if they are identifiable in the induced subgraph $G_{V\setminus \{ 2\}}$.}
    \label{fig:induced_graph}
\end{figure}

For instance, in the DAG in Fig.~\ref{fig:induced_graph}, the node 2 has only one path to its out-neighbor 5. According to Lemma~\ref{lemma:removal_node}, if the node 2 is measured and $f_{5,2}$ is identifiable, the edges $f_{5,1}$, $f_{5,3}$ and $f_{5,4}$ and the functions $F_1$, $F_3$ and $F_4$ are identifiable with the measurement of the node 5 if they are identifiable in the subgraph $G_{V\setminus \{ 2\}}$ obtained with the removal of the \linebreak node 2.

Lemma~\ref{lemma:removal_node} also holds in the linear case, but its importance lies mainly in the nonlinear case where the information associated with the measurement of the node $j$ can be obtained without measuring this node as it will be seen in the proof of Theorem~\ref{thm:DAG_Volterra}. Notice that the statement of Lemma~\ref{lemma:removal_node} is only local at a specific node $i$, but it can be  applied independently to each measured node.
Before introducing the main result of this section, we will need the following technical lemmas where we remind that $\tau$ is the shift operator of coordinates defined in Section~\ref{sec:notation}.

\begin{lemma}[Sum of multivariate functions]\label{lemma:multivariate_sum}
Given three non identically zero twice continuously differentiable functions $f:\R^p\to \R$ and $g,\tilde g:\R^q\to\R$. Let us assume that at least one of the partial derivatives of $f$ is not constant and they satisfy:
$$
f(\zero)= g(\zero)= \tilde g(\zero)=0;
$$
\vspace{-7mm}
\begin{multline}\label{eq:statement_lemma_function_h}
f(x_1\!+\!g(\y_1,\ldots,\y_r),\ldots,x_p\!+\!g(\tau^{p-1}\y_1,\ldots,\tau^{p-1}\y_r))=\\     
f(x_1\!+\!\tilde g(\y_1,\ldots,\y_r),\ldots,x_p\!+\!\tilde g(\tau^{p-1}\y_1,\ldots,\tau^{p-1}\y_r))+h,
\end{multline}
for all $x_i$ with $i\in \{1,\ldots,p\}$, and for all $\y_j$ with $j\in \{1,\ldots,r\}$, where $h$ is an arbitrary function that does not depend on $x_1,\ldots,x_p$.
Then $g=\tilde g$.
\end{lemma}
\begin{proof}
    The proof is deferred to Appendix~\ref{app:2}
\end{proof}
\begin{lemma}\label{lemma:recursively}
 In a DAG and in the class $\Ftwo$, if the function $F_i$ is identifiable, then all the functions $f_{i,j}$ and 
$F_j$ are identifiable for all $j\in \NN_i$.
\end{lemma}
\begin{proof}
    If $F_i$ corresponds to a source, the result is trivial. Let us assume that $F_i$ corresponds to a node that is not a source. 
    The output of any node $i$ is of the form:
    \begin{align}
        y_i^k&=u_i^{k-1}+F_i\nonumber\\
        &=u_i^{k-1}\!+\!\sum_{j\in\NN_i}\!f_{i,j}(u_{j}^{k-2}\!+\!F_{j}^{(1)},\ldots,u_{j}^{k-m_{j}-1}\!+\!F_{j}^{(m_{j})}).\nonumber
    \end{align}
    Let us assume that
    there exists a set $\{\tilde f\}\neq \{ f\}$ such that $F_i=\tilde F_i$, which implies:
\begin{multline}\label{eq:equal_f_tilde_DAG}
        \sum_{j\in\NN_i}f_{i,j}(u_{j}^{k-2}+F_{j}^{(1)},\ldots,u_{j}^{k-m_{j}-1}+F_{j}^{(m_{j})})=\\
        \sum_{j\in\NN_i}\tilde f_{i,j}(u_{j}^{k-2}+\tilde F_{j}^{(1)},\ldots,u_{j}^{k-m_{j}-1}+\tilde F_{j}^{(m_{j})}).
    \end{multline}
By Lemma~\ref{lemma:unique_path}, there is at least one in-neighbor $\ell$ with only one path to $i$. 
Then, we can use Lemma~\ref{lemma:sinks} to guarantee $f_{i,\ell}=\tilde f_{i,\ell}$, and \eqref{eq:equal_f_tilde_DAG} can be expressed as:
\begin{multline}\label{eq:fij_lemma5}
    f_{i,\ell}(u_{\ell}^{k-2}+F_{\ell}^{(1)},\ldots,u_{\ell}^{k-m_{\ell}-1}+F_{\ell}^{(m_{\ell})})=\\
         f_{i,\ell}(u_{\ell}^{k-2}+\tilde F_{\ell}^{(1)},\ldots,u_{\ell}^{k-m_{\ell}-1}+\tilde F_{\ell}^{(m_{\ell})})+\phi,
\end{multline}
where $\phi$ does not depend on any of the inputs $u_\ell$ since the only path from $\ell$ to $i$ is the edge $(i,\ell)$. 
By using Lemma~\ref{lemma:multivariate_sum} in \eqref{eq:fij_lemma5} we obtain $F_{\ell}^{(1)}=\tilde F_{\ell}^{(1)}$, which corresponds to the mathematical constraint associated with the measurement of the node $\ell$.
Then, the equality \eqref{eq:equal_f_tilde_DAG} becomes:
\begin{multline}\label{eq:node_i2_Volterra}
\sum_{j\in\NN_i\setminus\brc{\ell}}f_{i,j}(u_{j}^{k-2}+F_{j}^{(1)},\ldots,u_{j}^{k-m_{j}-1}+F_{j}^{(m_{j})})=\\
\sum_{j\in\NN_i\setminus\brc{\ell}}\tilde f_{i,j}(u_{j}^{k-2}+\tilde F_{j}^{(1)},\ldots,u_{j}^{k-m_{j}-1}+\tilde F_{j}^{(m_{j})}),
\end{multline}
and according to Lemma~\ref{lemma:removal_node}, the edges $f_{i,j}$ and functions $F_j$ for $j\in\NN_i\setminus\{\ell\}$ are identifiable if they are identifiable in the induced subgraph $G_{V\setminus\{ \ell\}}$.
By Lemma~\ref{lemma:unique_path}, in the subgraph $G_{V\setminus\{ \ell\}}$ there must be now another node with only one path to the node $i$ and we can follow a similar approach. Then, we can show that 
$$
f_{i,j}=\tilde f_{i,j} \quad \text{ for all } j\in\NN_i,
$$
and similarly for the functions
$$
F_j^{(1)}=\tilde F_j^{(1)} \quad \text{ for all } j\in\NN_i.
$$
\end{proof}

\begin{theorem}[DAGs]\label{thm:DAG_Volterra}
A DAG is identifiable in the class $\Ftwo$, if and only if all the sinks are measured.
\end{theorem}
\begin{proof}
    By Proposition~\ref{prop:sinks_sources} the measurement of all the sinks is necessary. Let us consider an arbitrary DAG and an arbitrary sink $i$. The measurement of the sink $i$ provides the identifiability of $F_i$ and according to Lemma~\ref{lemma:recursively}, all the functions $f_{i,j}$ and $F_j$ are identifiable for all $j\in\NN_i$. Then, we can apply Lemma~\ref{lemma:recursively} again to each function $F_j$ for $j\in\NN_i$, and by recursively  using Lemma~\ref{lemma:recursively}, we guarantee that $\{ f \}=\{ \tilde f \}$ for every path that ends in the sink $i$, and since in a DAG all the paths end in a sink \cite{bang2008digraphs}, by measuring the other sinks we can identify all the network. 
\end{proof}
We note that  Theorem~\ref{thm:DAG_Volterra} allows us to analyze the case of nonlinear functions $f_{i,j}$ that depend on outputs with zero delays (i.e., $y_j^{k-0}$), and this will be used in the proof of Theorem~\ref{thm:general_digraph}.
Since the class $\Ftwo$ excludes pure linear functions, the identifiability conditions of Theorem~\ref{thm:DAG_Volterra} are weaker than the linear case where to identify the outgoing edges of a node $i$, it is necessary to have $\abs{\NN_i^{out}}$ vertex-disjoint paths to measured nodes (see Corollary~V.2 in \cite{hendrickx2019identifiability}). In the nonlinear case, we only need to measure the sinks, independently of the number of out-neighbors of each node.  For instance, the DAG in Fig.~\ref{fig:bridge_graph} could represent a hydraulic network with 4 tanks (nodes) and nonlinear pipes (edges) where each tank receives a direct inflow from the outside ($u_i$) and indirect inflows from other tanks through internal pipes. The edges functions $f_{i,j}$ are used to model nonlinear phenomena like nonlinear resistance, a valve, an orifice, etc. For example, the function $f(x)=\tanh(x)$ can model saturation in flows and belongs to the class $\Ftwo$. According to Theorem~\ref{thm:DAG_Volterra}, we only need to measure the outflow of the tank 4 (sink of the DAG) to identify all the nonlinear functions in the system, even if the Tank 1 has two different output pipes (i.e., node 1 has two out-neighbors).
The vertex-disjoint paths do not play a role in the identifiability conditions in the nonlinear case since the derivative of the functions allow us to distinguish the information coming from different paths. Furthermore, notice that the measurement of the sinks provides the information associated with the measurement of all the nodes in the DAG. However, unlike the linear case, the symmetric version of this result for full measurement and partial excitation does not hold. Namely, it is not sufficient to excite all the sources to identify all the network when all the nodes are measured \cite{vizuete2024partial}. Therefore, different identifiability conditions are required for the full measurement case with nonlinear functions.

\section{General digraphs}\label{sec:general_digraphs}

Now, we analyze general digraphs where cycles are allowed.
In this case, the delays $T_j$ in \eqref{eq:function_Fi} for $j\in\NN_i^p$ might not be finite, which would imply an infinite number of variables due to the past values of the inputs.

\begin{figure}[!t]
    \centering
    \begin{tikzpicture}
    [
roundnodes/.style={circle, draw=black!60, fill=black!5, very thick, minimum size=1mm},roundnode/.style={circle, dashed, draw=black!60, fill=red!15, very thick, minimum size=1mm}
]

\node[roundnodes](node1){1};
\node[roundnodes](node2)[right=2.5cm of node1]{2};

\draw[-{Classical TikZ Rightarrow[length=1.5mm]}] (node1) to[out=45, in=135, looseness=1] node [above,swap] { $f_{2,1}$} (node2);

\draw[-{Classical TikZ Rightarrow[length=1.5mm]}] (node2) to[out=225, in=315, looseness=1] node [below,swap] { $f_{1,2}$} (node1);

\end{tikzpicture}
    \caption{Cycle graph with 2 nodes, where the function associated with the measurement of node 2 depends on an infinite number of past inputs due to the cycle.}
    \label{fig:cycle_graph}
\end{figure}
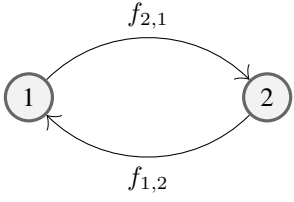

\begin{example}[Infinite inputs]\label{ex:infinite_delays}
    Let us consider the cycle graph in Fig.~\ref{fig:cycle_graph} with nonlinear functions of the form $f_{1,2}(y_2^{k-1})$, and $f_{2,1}(y_1^{k-1})$. The measurement of the node 2 provides an output of the form
    \begin{align}
        y_2^k&=u_2^{k-1}+f_{2,1}(u_1^{k-2}+f_{1,2}(y_2^{k-2}))\label{eq:NARMAX_model}\\
        &=u_2^{k-1}+f_{2,1}(u_1^{k-2}+f_{1,2}(u_2^{k-3}+f_{2,1}(y_1^{k-3}))),\nonumber
    \end{align}
    where we can observe that due to the cycles, the inputs will be repeated indefinitely with different delays, such that the function $F_2$ associated with the measurement is a function of an infinite number of variables. 
\end{example}

    Notice that the output \eqref{eq:NARMAX_model} corresponds to a NARX\footnote{Nonlinear autoregressive exogenous model.} \linebreak model \cite{nelles2020nonlinear}:
    $$
    y_2^k=F_{\mathrm{NARX}}(u_1^{k-2},u_2^{k-1},y_2^{k-2}),
    $$
    where past outputs are considered in the formulation.
    Nevertheless, in this work, we restrict our analysis to NFIR type models of the form \eqref{eq:function_Fi} that depend only on excitation signals (i.e., inputs) of the system.

Given that a general digraph can have cycles that depend on an infinite number of past values of the inputs, the design of experiments to obtain the ideal information associated with $F_i$ becomes more complex due to the infinite amount of information (past inputs) required.
To avoid this case and obtain identifiability conditions that could match realistic experimental settings, we make the following assumption, that is commonly used in system identification \cite{malti2006tutorial,khosravi2023kernel,pillonetto2023full,nelles2020nonlinear}.

\begin{assumption}[Network initially at rest]\label{ass:finite_variables}
The network is initially at rest at time $k=0$, which implies $y_i^k=0$ for all $k\le 0$  and all $i\in V$.
\end{assumption}
 Assumption~\ref{ass:finite_variables} covers the case of many networks that should be at rest at some time instant if external signals are no longer applied. According to our setting, this can be easily done by setting the values of all the excitation signals to zero and waiting an appropriate period of time so that the system is at rest again. Also, it is reasonable to assume that any experimental setting will consider the system at rest to begin with the application of the excitation signals, such that all the information obtained through the measured nodes is exclusively generated by the excitation signals, neglecting the influence of uncontrolled inputs. Moreover, Assumption~\ref{ass:finite_variables} is usually considered in the identification process to avoid the computation of cumbersome initial conditions \cite{campi2002finite}.

Under Assumption~\ref{ass:finite_variables}, when there are cycles in the network, the number of inputs that appear in the function associated with the measurement of a node is finite because the excitation signals can be nonzero only for $k>0$. This implies a truncation of the function $F_i$ in \eqref{eq:function_Fi} such that the output of a node $i$ is given by
\begin{multline}\label{eq:function_Fi_cycles}
    y_i^k=u_i^{k-1}+F_i^k(u_1^{k-2},\ldots,u_1^{0},\ldots,u_{n_i}^{k-2},\ldots,u_{n_i}^{0}),\\
    1,\dots,n_i\in \mathcal{N}_i^p,
\end{multline}
where $F_i^k$ depends on inputs only until the time $k=0$ (i.e., $u_i^0$). 
For instance, let us consider again the cycle graph in Fig.~\ref{fig:cycle_graph} with the functions of the Example~\ref{ex:infinite_delays} and the measurement of the node 2 at the time instant $k=4$. Under Assumption~\ref{ass:finite_variables}, the output is given by

\vspace{-3mm}

\small
\begin{align*}
        y_2^k&=u_2^{k-1}+F_2^k\\
        &=u_2^{k-1}+f_{2,1}(u_1^{k-2}+f_{1,2}(u_2^{k-3}+f_{2,1}(u_1^{k-4}+f_{1,2}(y_2^{k-4}))))\\
        &=u_2^{k-1}+f_{2,1}(u_1^{k-2}+f_{1,2}(u_2^{k-3}+f_{2,1}(u_1^{k-4}))),
    \end{align*}
\normalsize
which depends on a finite number of inputs since $y_2^{k-4}=0$. 

Unlike the case of DAGs, the functions $F_i^k$ in \eqref{eq:function_Fi_cycles} are parameterized by $k$. In this case we assume that for each $k$, when we measure a node $i$, we know perfectly the corresponding function $F_i^k$. Since we have a family of functions parameterized by $k$, the identifiability of edges must be guaranteed for each $F_i^k$. We extend the Definitions~\ref{def:set_measured_functions} and \ref{def:math_identifiability} for the family of functions.
\begin{definition}[Set of measured functions at $k$]\label{def:new_measured_nodes}
    Given a set of measured nodes  $\mathcal{N}^m$, the set of measured functions $F^k(\mathcal{N}^m)$ associated with $\mathcal{N}^m$  at a time instant $k$ is given by:
    $$
    F^k(\mathcal{N}^m):=\{F_i^k\;|\;i\in \mathcal{N}^m\}.
    $$
\end{definition}

The set $F^k$ represents how the outputs of measured nodes $y^k$ depend on the excitation signals from 0 to $k$.

\begin{definition}[Identifiability at $k$]\label{def:new_math_identifiability}
    Given a set of functions $\{ f \}=\{f_{i,j}\in \F\;|\;(i,j)\in E\}$ that generates $F^k(\NN^m)$ and another set of functions $\{ \tilde f \}=\{\tilde f_{i,j}\in\F\;|\;(i,j)\in E\}$ that generates $\tilde F^k(\NN^m)$. An edge $f_{i,j}$ is identifiable in a class $\F$ if there exists a $k$ such that $
    F^k(\NN^m)=\tilde F^k(\NN^m)
    $ implies $f_{i,j}=\tilde f_{i,j}$. A network $G$ is identifiable in a class $\F$ if there exists a $k$ such that 
    $
    F^k(\NN^m)=\tilde F^k(\NN^m),
    $
    implies
    $
    \{ f \}=\{ \tilde f \}
    $.
\end{definition}

Unlike DAGs where nonzero excitation signals $u_i^k$ at time instants $k<0$ are allowed, for general digraphs, the time instant considered for the measurement of a node should be long enough to let all the excitation signals to have an effect on the output of the measured node. If we consider again the cycle graph in Fig.~\ref{fig:cycle_graph}, but with a measurement of the node 2 performed at the time instant $k=2$, the output is given by:
\begin{align}
    y_2^k&=u_2^{k-1}+F_2^k\nonumber\\
    &=u_2^{k-1}+f_{2,1}(u_1^{k-2}+f_{1,2}(y_2^{k-2}))\nonumber\\
    &=u_2^{k-1}+f_{2,1}(u_1^{k-2}),\label{eq:output_F3_truncated}
\end{align}
since $y_2^{k-2}=0$. Even if the measurement of the node 2 could give us information also about the function $f_{1,2}$, this is not reflected on the function $F_2^k$ due to the short period of time chosen for the measurement. In this case, any function $f_{1,2}\in \F$ could generate the output \eqref{eq:output_F3_truncated}, which implies that $f_{1,2}$ is not identifiable at the time instant $k=2$.

In any digraph, we will be able to consider
\begin{equation}\label{eq:k_min}
\bar k_{\min}=m_{\max}+\text{diam}(G),    
\end{equation}
where $\text{diam}(G)$ is the diameter of the digraph $G$ and $m_{\max}$ is the largest delay of the nonlinear functions $f_{i,j}$ among all $i$ and any $j$. By choosing $k_{\min}=\bar k_{\min}$ we will guarantee the appearance of all the possible nonlinear functions $f_{i,j}$ with all their arguments in $F^{k_{\min}}(\NN^m)$ (see the proof of Lemma~\ref{lemma:equivalence}). 
Notice that the consideration of time instants $k> \bar k_{\min}$ for the analysis of identifiability is not restrictive since it is logical that in any experimental setting, we need to wait until the outputs of measured nodes reflect the influence of all the possible excitation signals. Also, since the definition of the time $k=0$ depends on the context of the analysis, it can be associated with the beginning of the application of the excitation signals in the nodes.

 Definition~\ref{def:set_measured_functions} for DAGs is encompassed by Definition~\ref{def:new_measured_nodes} since the information obtained from the measurement of a node (i.e., $F(\NN^m)$) is independent of the time instant $k$. Similarly, Definition~\ref{def:math_identifiability} is encompassed by Definitions~\ref{def:new_math_identifiability}, since for DAGs, identifiability at an arbitrary time instant $k$ guarantees identifiability for all time instants $k$ because the functions $F_i$ are not parameterized \linebreak by $k$.

During the analysis of identifiability, one of the most basic steps that guarantees the identification of all the network is the identifiability of the in-neighbors of measured nodes. This has been performed by isolating parts of the network with appropriate zero excitation inputs (see Lemma~\ref{lemma:sinks}). However, the extension of this methodology to general digraphs is not always possible since several inputs can be repeated due to the cycles.
For example, let us consider again the cycle graph in Fig.~\ref{fig:cycle_graph} with the functions $f_{2,1}(y_1^{k-1},y_1^{k-3})$ and $f_{1,2}(y_2^{k-1})$  where the measurement of the node 2 at the time instant $k=4$ provides the output:
\begin{align*}
    y_2^k&=u_2^{k-1}+f_{2,1}(y_1^{k-1},y_1^{k-3})\\
    &=u_2^{k-1}+f_{2,1}(u_1^{k-2}+f_{1,2}(u_2^{k-3}+f_{2,1}(u_1^{k-4})),u_1^{k-4}).
\end{align*}
Notice that we cannot identify $f_{2,1}$ by setting to zero the other inputs because the input $u_1^{k-4}$ is repeated due to the cycle. In this way, it is not possible to freely select the nonzero excitation signals for the analysis of a measured function $F_i$ and this is the main technical difficulty of the identifiability of graphs with cycles. In the linear case, this problem has been handled by measuring additional nodes of the network \cite{shi2021single}. Nevertheless, in the nonlinear case, the results can still be extended without measuring more nodes by considering additively separable functions.

\begin{assumption}[Additively separable functions]\label{ass:separable}
The nonlinear functions $f_{i,j}$ are additively separable:
\begin{equation}\label{eq:separation_analytic}
  f_{i,j}(y_j^{k-1},\ldots,y_j^{k-m_j})=\sum_{\ell=1}^{m_j}
  f_{i,j}^{\{\ell\}}(y_j^{k-\ell}),  
\end{equation} 
where for each function $f_{i,j}^{\{\ell\}}$, if it is nonzero, then it belongs to $\Ftwo$. 
\end{assumption}

The determination of an initial time and a final time for the identification process implies in some way that nodes associated with inputs $u^0$ stop being influenced through functions $f_{i,j}^{\{ \ell \}}$ associated with incoming edges at negative time instants, and nodes associated with inputs $u^k$ cease to influence other nodes through functions $f_{i,j}^{\{ \ell \}}$ associated with outgoing edges. This clearly entails an analogy with the notions of sources and sinks of DAGs, and this allows us to construct a special type of unfolded digraph associated with any digraph that will be used in the proofs of the results. This type of unfolded digraph is commonly used in dynamic Bayesian networks \cite{murphy2002dynamic} and Petri nets \cite{mcmillan1993using}, and its topology will be generated based on nodes with a nonempty set of out-neighbors (i.e., $\NN_i^{\text{out}}\neq \emptyset$), where only for its construction, we will assume that the first node denoted by $i_0$ satisfies $\NN_{i_0}^{\text{out}}\neq\emptyset$. To distinguish the nodes of the unfolded digraph from the nodes of the original digraph, we will use a subscript $i_m$ where $m$ will denote the time instant of creation of the copy of the node $i$.
Intuitively, in the unfolded digraph, there is a copy of each node for each time instant $t=1,\ldots,k-1$, and two nodes $v_t$ and $w_m$ are connected if the node $v_t$ at time $t$ influences directly the node $w_m$ at time $m$.
Notice that the unfolded digraph $H_i^k(G)$ is composed by several copies of all the nodes in $G$. However, some of the copies are disconnected (no incoming or outgoing edges) and their effect will be neglected in the analysis of identifiability. In Fig.~\ref{fig:general_digraph} we present the unfolded digraph $H_i^k(G)$ of a general digraph $G$ with several cycles and nonlinear functions of the form $f_{i,j}(y_j^{k-1})$ where we can see the several copies of the nodes grouped according to the delays of the excitation signals (each group corresponds to a different set of a multipartite digraph). The disconnected copies of nodes have been removed from $H_i^k(G)$, since they have no effect on the identifiability problem. 
Notice that the unfolded digraph $H_i^k(G)$ is unique and different for each measured \linebreak node $i$.

\begin{algorithm}[!t]\label{algo:adapted_tree}
    \caption{Unfolded digraph $H_i^k(G)$}
    \KwData{Graph $G=(V,E)$, measured node $i$, measurement time $k$}
    \KwResult{Unfolded digraph $H_i^k(G)=(V_H,E_H)$ }

\color{black}

Create a copy of the node $i$ denoted by $i_0\in V_H$ with excitation signal $u_i^k$

\For{$t=1,\ldots,k-1$}{

    Create copies of the nodes $v\in V$ denoted by $v_t\in V_H$ with excitation signal $u_v^{k-t}$
    
    \For{each node $v_t\in V_H$ and each node $w_m\in V_H$ with $\NN_{w_m}^{out}\neq \emptyset$ and $m=0,\ldots,t-1$}{

            \If {$f_{w,v}^{\{t-m\}}(y_v^{k-t+m})\neq 0$ }{
                Connect $w_m$ and $v_t$ with $f_{w,v}^{\{t-m\}}(y_v^{k})$ located in the edge $(w_m,v_t)\in E_H$
               
            }
}

        }
    
\end{algorithm}

\begin{definition}[Unfolded digraph]
    Under Assumption~\ref{ass:separable}, an unfolded digraph at a node $i$ and the time instant $k$   of a digraph $G$ denoted by $H_i^k(G)$ is a digraph generated through Algorithm~\ref{algo:adapted_tree}
    where the nonlinear functions of the edges have delay 0.
\end{definition}

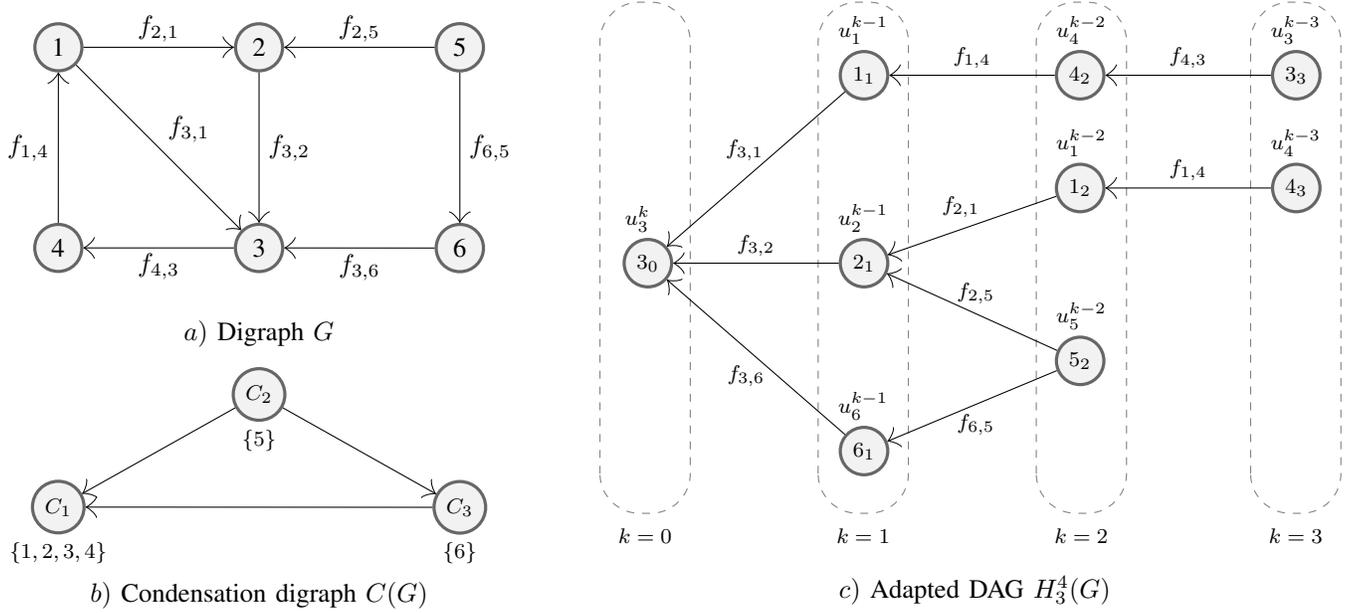
\begin{figure*}
    \centering
    \begin{tikzpicture}
    [
roundnodes/.style={circle, draw=black!60, fill=black!5, very thick, minimum size=1mm},roundnode/.style={circle, draw=white!60, fill=white!5, very thick, minimum size=1mm},roundnodes2/.style={circle, draw=black!60, fill=black!5, very thick, minimum size=0.6cm},roundnodes3/.style={circle, dashed, draw=black!60, fill=red!15, very thick, minimum size=0.6cm},roundnodes4/.style={circle, dashed, very thick, minimum size=0.6cm}
]
\node[roundnodes](node1){1};
\node[roundnodes](node2)[right=of node1,yshift=0mm,xshift=1cm]{2};
\node[roundnodes](node3)[below=of node2,yshift=-1cm,xshift=0cm]{3};
\node[roundnodes](node4)[left=of node3,yshift=0mm,xshift=-1cm]{4};
\node[roundnodes](node5)[right=of node2,yshift=0mm,xshift=1cm]{5};
\node[roundnodes](enode1)[below=of node5,yshift=-1cm,xshift=0cm]{6};
\node(graph)[below=of node3,yshift=0.5cm,xshift=0cm]{$a) $ Digraph $G$};

\draw[-{Classical TikZ Rightarrow[length=1.5mm]}] (node1) to node [swap,yshift=2.5mm] {$f_{2,1}$} (node2);
\draw[-{Classical TikZ Rightarrow[length=1.5mm]}] (node2) to node [right,swap] {$f_{3,2}$} (node3);
\draw[-{Classical TikZ Rightarrow[length=1.5mm]}] (node3) to node [swap,yshift=-2.5mm] {$f_{4,3}$} (node4);
\draw[-{Classical TikZ Rightarrow[length=1.5mm]}] (node4) to node [left,swap] {$f_{1,4}$} (node1);
\draw[-{Classical TikZ Rightarrow[length=1.5mm]}] (node1) to node [right,swap,yshift=2.5mm] {$f_{3,1}$} (node3);
\draw[-{Classical TikZ Rightarrow[length=1.5mm]}] (node5) to node [swap,yshift=2.5mm] {$f_{2,5}$} (node2);
\draw[-{Classical TikZ Rightarrow[length=1.5mm]}] (node5) to node [right,swap] {$f_{6,5}$} (enode1);
\draw[-{Classical TikZ Rightarrow[length=1.5mm]}] (enode1) to node [swap,yshift=-2.5mm] {$f_{3,6}$} (node3);

\footnotesize
\node[roundnodes2](cnode1)[below=of node4,yshift=-1.75cm,xshift=0cm]{$C_1$};
\node(C1)[below=of cnode1,yshift=1cm,xshift=0cm]{$\{ 1,2,3,4\}$ };
\node[roundnodes2](cnode2)[below=of node3,yshift=-0.25cm,xshift=0cm]{$C_2$};
\node(C2)[below=of cnode2,yshift=1cm,xshift=0cm]{$\{ 5\}$ };
\node[roundnodes2](cnode3)[below=of enode1,yshift=-1.75cm,xshift=0cm]{$C_3$};
\node(C3)[below=of cnode3,yshift=1cm,xshift=0cm]{$\{ 6\}$ };

\normalsize

\node(graph1)[below=of cnode2,yshift=-1cm,xshift=0cm]{$b) $ Condensation digraph $C(G)$};

\draw[-{Classical TikZ Rightarrow[length=1.5mm]}] (cnode2) to node [right,swap,yshift=2.5mm] {} (cnode1);
\draw[-{Classical TikZ Rightarrow[length=1.5mm]}] (cnode2) to node [right,swap,yshift=2.5mm] {} (cnode3);
\draw[-{Classical TikZ Rightarrow[length=1.5mm]}] (cnode3) to node [right,swap,yshift=2.5mm] {} (cnode1);

\footnotesize

\node[roundnodes2](node6)[below=of node5,yshift=-1.2cm,xshift=2.5cm]{$3_0$} (7.7,-2.55) node[above] {$u_3^{k}$};
\node[roundnodes2](node7)[right=of node6,yshift=2.5cm,xshift=1.2cm]{$1_1$} (10.7,-0.05) node[above] {$u_1^{k-1}$};
\node[roundnodes2](node8)[right=of node6,yshift=-2.5cm,xshift=1.2cm]{$6_1$} (10.7,-5.05) node[above] {$u_6^{k-1}$};
\node[roundnodes2](node9)[right=of node6,yshift=0cm,xshift=1.2cm]{$2_1$} (10.7,-2.55) node[above] {$u_2^{k-1}$};
\node[roundnodes2](node10)[right=of node8,yshift=1.2cm,xshift=1.2cm]{$5_2$} (13.6,-3.85) node[above] {$u_5^{k-2}$};
\node[roundnodes2](node11)[right=of node9,yshift=1.0cm,xshift=1.2cm]{$1_2$} (13.6,-1.55) node[above] {$u_1^{k-2}$};
\node[roundnodes2](node12)[right=of node7,yshift=0cm,xshift=1.2cm]{$4_2$} (13.6,-0.05) node[above] {$u_4^{k-2}$};
\node[roundnodes2](node14)[right=of node12,yshift=0.0cm,xshift=1.2cm]{$3_3$} (16.45,-0.05) node[above] {$u_3^{k-3}$};
\node[roundnodes2](enode2)[right=of node11,yshift=0cm,xshift=1.2cm]{$4_3$} (16.45,-1.55) node[above] {$u_4^{k-3}$};

\draw[black!50,dashed,rounded corners=15pt]
  (7.2,-6.2) rectangle ++(1.2,6.8);
  \draw[black,dashed,rounded corners=15pt] (7.8,-6.7) node[above] {$k=0$} ;
  \draw[black!50,dashed,rounded corners=15pt]
  (10.1,-6.2) rectangle ++(1.2,6.8) (0,0);
  \draw[black,dashed,rounded corners=15pt] (10.7,-6.7) node[above] {$k=1$} ;
  \draw[black!50,dashed,rounded corners=15pt]
  (13,-6.2) rectangle ++(1.2,6.8) (0,0);
  \draw[black,dashed,rounded corners=15pt] (13.6,-6.7) node[above] {$k=2$} ;
  \draw[black!50,dashed,rounded corners=15pt]
  (15.85,-6.2) rectangle ++(1.2,6.8) (0,0);
  \draw[black,dashed,rounded corners=15pt] (16.45,-6.7) node[above] {$k=3$} ;

\draw[-{Classical TikZ Rightarrow[length=1.5mm]}] (node7) to node [left,swap,xshift=2mm,yshift=2.5mm] {$f_{3,1}$} (node6);
\draw[-{Classical TikZ Rightarrow[length=1.5mm]}] (node9) to node [above,swap,xshift=0mm,yshift=0mm] {$f_{3,2}$} (node6);
\draw[-{Classical TikZ Rightarrow[length=1.5mm]}] (node8) to node [left,swap,xshift=2.2mm,yshift=-2.5mm] {$f_{3,6}$} (node6);
\draw[-{Classical TikZ Rightarrow[length=1.5mm]}] (node10) to node [left,swap,xshift=4mm,yshift=2.5mm] {$f_{2,5}$} (node9);
\draw[-{Classical TikZ Rightarrow[length=1.5mm]}] (node10) to node [left,swap,xshift=4mm,yshift=-2.5mm] {$f_{6,5}$} (node8);
\draw[-{Classical TikZ Rightarrow[length=1.5mm]}] (node11) to node [left,swap,xshift=2mm,yshift=2.5mm] {$f_{2,1}$} (node9);
\draw[-{Classical TikZ Rightarrow[length=1.5mm]}] (node12) to node [above,swap,xshift=0mm,yshift=0mm] {$f_{1,4}$} (node7);
\draw[-{Classical TikZ Rightarrow[length=1.5mm]}] (node14) to node [above,swap,xshift=0mm,yshift=0mm] {$f_{4,3}$} (node12);
\draw[-{Classical TikZ Rightarrow[length=1.5mm]}] (enode2) to node [above,swap,xshift=0mm,yshift=0mm] {$f_{1,4}$} (node11);

\normalsize

\node(graph2)[below=of node11,yshift=-3.7cm,xshift=-1.4cm]{$c) $ Unfolded digraph $H_3^4(G)$};

\end{tikzpicture}
    \caption{General digraph $G$ with its corresponding condensation digraph $C(G)$ and unfolded digraph $H_3^4(G)$ constructed at the node 3 and the time instant $k=4$. The identifiability of $G$ at $k=4$ is equivalent to the identifiability of $H_3^4(G)$ where several copies of the nodes of $G$ are created for its construction with inputs determined by the corresponding delays.
    The unfolded digraph is also a multipartite digraph with sets $\{ 3_0\}$, $\{ 1_1,2_1,6_1\}$, $\{ 4_2,1_2,5_2\}$ and $\{ 3_3,4_3\}$, 
    which are constructed by grouping nodes with the same delays on the inputs. 
    }
    \label{fig:general_digraph}
\end{figure*}

The unfolded digraph can also be seen as a multipartite digraph where each set is composed by the nodes with excitation signals having the same delays. Since the edges can only go from a node to another node with a smaller delay, there are no cycles and the unfolded digraph is clearly a DAG, which allows us to analyze the identifiability of the network in a framework similar to Section \ref{sec:DAG}. The utility of this mathematical construction is highlighted by the following lemma.

\begin{lemma}\label{lemma:equivalence}
Given a digraph $G$ and a measured node $i$. Under Assumptions~\ref{ass:finite_variables} and \ref{ass:separable}, for $i\in \NN^m$ and any $k>\bar k_{\min}$, the edges on all the walks that arrive to the node $i$ in $G$ are identifiable if the unfolded digraph $H_i^k(G)$ is identifiable.
\end{lemma}
\begin{proof}
    Under Assumption~\ref{ass:finite_variables}, the function $F_i^k$ associated with the measurement of a node $i$ depends on a finite number of inputs.
    The measurement of the node $i$ in $G$ provides an output of the form:
\begin{align}
y_i^k&=u_i^{k-1}+F_i^k\nonumber\\
&=u_i^{k-1}+\sum_{j\in\NN_i}
f_{i,j}(y_j^{k-1},\ldots,y_j^{k-m_j})\nonumber\\
&=u_i^{k-1}+\sum_{j\in\NN_i}\sum_{\ell_j=1}^{\hat m_j} f_{i,j}^{\{\ell_j\}}(y_1^{k-\ell_j})\nonumber\\
&=u_i^{k-1}+\sum_{j\in\NN_i}\sum_{\ell_j=1}^{\hat m_j} f_{i,j}^{\{\ell_j\}}\prt{u_{j}^{k-\ell_j-1}+F_{j}^{(\ell_j)}}.\label{eq:intermediate_separable}
\end{align}
Since $k>\bar k_{\min}$, we guarantee that $\hat m_j=m_j$ such that all the possible nonlinear functions $f_{i,j}^{\{\ell_j\}}$ appear in \eqref{eq:intermediate_separable}.  
   Let us assume that there exists a set $\{\tilde f\}\neq \{ f\}$ such that $F_i^k=\tilde F_i^k$, which implies:
\begin{multline}\label{eq:intermediate_separable_2}
\sum_{j\in\NN_i}\sum_{\ell_j=1}^{m_j} f_{i,j}^{\{\ell_j\}}\prt{u_{j}^{k-\ell_j-1}+F_{j}^{(\ell_j)}}=\\
\sum_{j\in\NN_i}\sum_{\ell_j=1}^{m_j} \tilde f_{i,j}^{\{\ell_j\}}\prt{u_{j}^{k-\ell_j-1}+\tilde F_{j}^{(\ell_j)}}.
\end{multline}
    Now, let us consider the unfolded digraph $H_i^k(G)$. The measurement of the sink $i$ in $H_i^k(G)$ provides the output:
    \begin{equation}\label{eq:node_i_adapted_tree}
    y_i^k=u_i^{k-1}+\sum_{j\in\NN_i}\sum_{\ell_j=1}^{m_j} f_{i,j}^{\{\ell_j\}}\prt{u_{j}^{k-\ell_j-1}+\hat F_{j}^{(\ell_j)}}.    
    \end{equation}
    Since the unfolded digraph $H_i^k(G)$ guarantees that each node is connected by copies of the same in-neighbors in $G$ where the delays of the excitation signals are given by the delays of the nonlinear functions and the time instants in which the nodes were created, we can guarantee that the outputs of all the nodes in $G$ and $H_i^k(G)$ are the same with the corresponding delays, which implies $\hat F_j^{(\ell_j)}=F_j^{(\ell_j)}$ and hence $\hat F_j=F_j$. 
    Then, \eqref{eq:node_i_adapted_tree} implies \eqref{eq:intermediate_separable_2}. Since $k>\bar k_{\min}$, we have that for each $F_j$, all the possible nonlinear functions $f_{j,p}^{\{\ell_p\}}$ associated with an in-neighbor $p$ of $j$ appear in $F_j$.  Hence, if the nonlinear functions of the edges are identifiable in $H_i^k(G)$, they are also identifiable in $G$.
\end{proof}

The following theorem provides identifiability conditions for any digraph $G$ based on the measurement of nodes associated with the condensation digraph $C(G)$. Since the condensation digraph $C(G)$ is a DAG, it will allow us to choose the measured nodes (sinks) such that the information of all the walks in $G$ appear in $F(\NN^m)$. Then, the unfolded digraph $H_i^k$ is used to prove the identifiability of all the walks that end at a measured node $i$ chosen according to $C(G)$. 

\begin{theorem}[General digraph]\label{thm:general_digraph}
Under Assumptions~\ref{ass:finite_variables} and \ref{ass:separable}, for any $k> \bar k_{\min}$, a digraph is identifiable in the class $\Ftwo$, if and only if one node is measured in every sink of the condensation digraph.
\end{theorem}
\begin{proof}
Let us consider an arbitrary digraph $G$. 
    By measuring one node of each sink of the condensation digraph $C(G)$, we can guarantee that there is a walk from any node of the network to one of the measured nodes. Let us assume that we measure one node $i$ of a sink of $C(G)$ at an arbitrary time instant $k> \bar k_{\min}$. According to Lemma~\ref{lemma:equivalence}, the walks that arrive to the node $i$ are identifiable if the unfolded digraph $H_i^k(G)$ is identifiable at $k>\bar k_{\min}$. By Theorem~\ref{thm:DAG_Volterra}, the unfolded digraph $H_i^k(G)$ is identifiable with the measurement of the sink corresponding to the node $i$, which implies that the edges of all the walks in $G$ that end in the node $i$ are identifiable, and since all the walks end in a sink of $C(G)$, we identify all the network at the time instant $k$. Since $k$ was arbitrary, the procedure can be applied for other $k>\bar k_{\min}$, such that the network is identifiable.
\end{proof}

The static model is a particular case of additively separable functions, where the function $f_{i,j}$  in \eqref{eq:separation_analytic} only depends \linebreak on $y_j^{k-1}$:
\begin{equation}\label{eq:nonlinear_model_static}
y_i^k=\sum_{j\in \mathcal{N}_i}f_{i,j}(y_j^{k-1})+u_i^{k-1}. 
\end{equation}
Therefore, Theorem~\ref{thm:general_digraph} allows us to complete the results of identifiability of networks in the static case analyzed in \cite{vizuete2023nonlinear}, where the case of general digraphs was missing.

\begin{remark}[Self-loops]
    In the case of nonlinear state space representations where dynamics at the level of nodes can be present as self-loops, the output of each node $i$ is given by:
\begin{equation}\label{eq:nonlinear_model_including_i}
y_i^k=\sum_{j\in \mathcal{N}_i^i}f_{i,j}(y_j^{k-1},y_j^{k-2},\ldots,y_j^{k-m_j})+u_i^{k-1}, 
\end{equation}
where $\NN_i^i$ is the set of in-neighbors of the node $i$ including $i$. Notice that in the associated unfolded digraph, the set of in-neighbors of a node $i$ is just extended to $\NN_i^i$ including the same node $i$ connected by $f_{i,i}$. Therefore, Theorem~\ref{thm:general_digraph} could also be used for this different framework. However, notice that digraphs with self-loops are no longer simple digraphs and that is why they are not consider in the model class of this work.  
\end{remark}

Similarly to the case of general DAGs, the identifiability conditions of Theorem~\ref{thm:general_digraph} are also weaker than in the linear case where the identifiability of $\abs{\NN_i^{out}}$ outgoing edges of a node $i$ requires $\abs{\NN_i^{out}}$ vertex-disjoint paths to measured nodes. In the nonlinear case, we only need to measure a node of each sink of the condensation digraph, independently of the number of out-neighbors of the nodes in the network. For instance, in the digraph $G$ in Fig.~\ref{fig:general_digraph}, it is sufficient to measure the node 3 that belongs to the sink of $C(G)$ to identify all the network, even if the node 5 has two out-neighbors. The role of the sink in the condensation digraph is related to the notion of the unfolded digraph, which is a DAG, used in the derivation of the identifiability conditions.

\section{Conclusions and future work}\label{sec:conclusions}

We have shown fundamental differences in the identifiability conditions for nonlinear networks with respect to the linear case. Specifically, we have shown for pure nonlinear functions that identifiability was determined by measuring the sinks in the case of DAGs, or nodes in the sinks of the condensation digraph for general digraphs assuming that the nonlinear functions are additively separable and the network is at rest in the last case. These identifiability conditions are weaker than in the linear case and are due to the fact that the nonlinearity of the functions allow us to distinguish the information coming to a node through different paths.

Our approach of the nonlinear network identifiability is novel, and there are consequently many open problems.
Clearly, the family of functions considered for the identifiability problem plays a crucial role. The first extension would be the analysis of general digraphs when the nonlinear functions are not necessarily additively separable. Also, in a more general framework, it would be important to analyze more general classes of functions that are not differentiable. 
Finally, the derivation of identifiability conditions for non-additive dynamical models and node dynamics \cite{vanelli2025local} are definitely interesting areas of research.

\bibliographystyle{IEEEtran}
\bibliography{TAC}

\appendix

\subsection{Proof of Lemma~\ref{lemma:multivariate}}\label{app:1}

First, we introduce some definitions and auxiliary results that will be used in the proof of Lemma~\ref{lemma:multivariate}.

We define the mapping $\Psi:\R^s\to \R^p$:
\begin{equation}
\Psi(\Y):=(\psi(\y_1,\ldots,\y_r),\ldots,\psi(\tau^{p-1}\y_1,\ldots,\tau^{p-1}\y_r)),    
\end{equation}
where $\psi=\tilde g-g$ and $\Y=(\y_1,\ldots,\y_r,\ldots)$.
\begin{lemma}\label{lemma:basis}
    Given a non-constant continuously differentiable function  $\psi:\R^q\to\R$. Then, $\mathrm{Im}(\Psi)$ spans the vector \linebreak space $\R^p$.
\end{lemma}
\begin{proof}
    Let us consider the Jacobian of the mapping $J_\Psi$.
    If $\psi$ is not constant, there must exist a variable $y_{j,\ell}$ such that $\frac{\partial \psi}{\partial y_{j,\ell}}\not\equiv 0$, where $y_{j,\ell}$ denotes the variable that corresponds to the position $\ell$ in the group of variables $\y_j$. Without loss of generality, let us assume that $y_{j,\ell}$ is the first variable such that $\frac{\partial \psi}{\partial y_{j,\ell}}\not\equiv 0$. Then, there exists a point $\Y^*=(\y_1^*,\ldots,\y_r^*,\ldots)$ such that $\frac{\partial \psi}{\partial y_{j,\ell}}(\Y^*)\not\equiv  0$, and for this point $\Y^*$, the Jacobian $J_ \Psi(\Y^*)$ contains the vector $(\alpha^*,0,\ldots,0)$ corresponding to the column of $\frac{\partial \psi}{\partial y_{j,\ell}}(\Y^*)$, and hence, the span of $J_\Psi$ contains $e_1=(1,0,\ldots,0)$. Since the Jacobian $J_\Psi$ is a Toeplitz matrix, for the point $\Y^{**}=(0,\y_1^*,\ldots,\y_r^*,\ldots)$, the Jacobian $J_\Psi(\Y^{**})$ contains the vector $(*,\alpha^*,0,\ldots,0)$, and hence, the span of $J_\Psi$ contains $e_2=(0,1,0,\ldots,0)$ given that it also contains $e_1$. By following this procedure, we can see that the span of $\mathrm{Im}(J_\Psi)$ contains all the standard basis $e_i$ of $\R^p$. Thus, $\mathrm{Im}(\Psi)$ spans the vector space $\R^p$.
\end{proof}

The following technical lemma will be used in the proof of Lemma~\ref{lemma:contains_ball}.

\begin{lemma}[Theorem 2.35 \cite{laczkovich2017real}]\label{lemma:mapping}
    Let $H\subset \R^p$ and let $f:H\to\R^q$, where $p\ge q$. If $f$ is continuously differentiable at the point $a\in \text{int} \;H$ and the Jacobian $J_f(a):\R^p\to\R^q$ is surjective, then the range of $f$ contains a neighborhood of $f(a)$.
\end{lemma}

\begin{lemma}\label{lemma:contains_ball}
    If $S=\mathrm{Im}( \Psi)$ is not contained in any vector space of dimension $p-1$, then the sumset $\underbrace{S+S+\cdots+S}_{p \text{ times}}$  contains a ball $B(c,\epsilon)$. 
\end{lemma}
\begin{proof}
    Let us define the mapping $\Theta(\Y_1,\ldots,\Y_p):=\Psi(\Y_1)+\Psi(\Y_2)+\cdots+\Psi(\Y_p)$.
    By definition of the sumset we have $\underbrace{S+S+\cdots+S}_{p \text{ times}}=\{ \Psi(\Y_1)+\Psi(\Y_2)+\cdots+\Psi(\Y_p)\}=\mathrm{Im}(\Theta)$. According to Lemma~\ref{lemma:basis}, there exists a point $(\Y_1^*,\ldots,\Y_p^*)$ such that $\{ \nabla \Psi(\Y_1^*),\ldots,\nabla\Psi(\Y_p^*)\}$ contains a basis of $\R^p$. Since at the point $(\Y_1^*,\ldots,\Y_p^*)$, the Jacobian   $J_ \Theta=(J_ \Psi(\Y_1^*)  \quad \cdots \quad J_ \Psi(\Y_n^*))$ is full rank, then by Lemma~\ref{lemma:mapping}, $\mathrm{Im}(\Theta)$ contains a ball  $B(c,\epsilon)$.  
\end{proof}

Finally, we state the proof of Lemma~\ref{lemma:multivariate}.

\begin{myproof}{Lemma~\ref{lemma:multivariate}}

Let us consider the change of variables:
    $$
    z_1=x_1+g(\y_1,\ldots,\y_r);
    $$
    $$
    \vdots
    $$
    $$
    z_p=x_p+g(\tau^{p-1}\y_1,\ldots,\tau^{p-1}\y_r).
    $$
    Then, \eqref{eq:equality_multivariate} can be written as:
    \begin{multline}\label{eq:app1_1}
    f(z_1,\ldots,z_p)=
        f(z_1+\psi(\y_1,\ldots,\y_r),\ldots,\\
        z_p+\psi(\tau^{p-1}\y_1,\ldots,\tau^{p-1}\y_r)),    
    \end{multline}
    \begin{equation}\label{eq:z_Psi}
        f(\z)=f(\z+\Psi(\Y)) \text{ for all } \z\in\R^p,
    \end{equation}
    where $\psi=\tilde g-g$ and $\z:=(z_1,\ldots,z_p)$. If $\psi\equiv0$, then it corresponds to one of the possible results of the Lemma. Now, let us assume that $\psi\not\equiv0$, and hence a non-constant function. From \eqref{eq:z_Psi} we have that:
    \begin{align}
        f(\z)&=f(\z+\Psi(\Y_1)+\Psi(\Y_2)+\cdots+\Psi(\Y_p))\nonumber\\&=f(\z+\Theta(\Y_1,\ldots,\Y_p)),
    \end{align}
    where according to Lemma~\ref{lemma:contains_ball}, the image of the mapping $\Theta$ contains a ball $B(c,\epsilon)$. This implies that for an arbitrary point $\z^*$, we get
    \begin{equation*}
    f(\z^*-c)=f(\z^*-c+\vv) \text{ for all } \vv\in B(c,\epsilon).  
    \end{equation*}
    Notice that any point $\w$ in the ball $B(\z^*,\epsilon)$ can  be written as $\z^*-c+\vv$ for some $\vv\in B(c,\epsilon)$. When $\vv=c$, we have $f(\z^*-c)=f(\z^*)$, and since the function $f(\z^*-c+\vv)$ is the same for all $\vv\in B(c,\epsilon)$, we obtain
    $$
    f(\z^*)=f(\w) \text{ for all } \w\in B(\z^*,\epsilon),
    $$
    which implies that $f$ is locally constant. 

   \noindent It remains to prove that the function $f$ is constant. Several versions are available in the literature, but we include our own version for the sake of completeness. Let us consider two arbitrary points $a$ and $b$ and the line segment joining them  $tb + (1-t)a$, for $t\in [0,1]$. The function $g(t)=f(tb + (1-t)a)$ is continuous and also locally constant. We will show by contradiction that $g(1)=g(0)$. Let us suppose that this is not the case and we take $t^* = \inf \{t: g(t)\neq f(a))$. By definition, $g(t) = f(a)$ for $t\in[0,t^*)$, and hence, also in 
   $t^*$ by continuity. Since $g$ is locally constant, the function must satisfy $g(t)=f(a)$ for $t\in (t^*-\epsilon$ , $t^* + \epsilon)$, which contradicts the definition of $t^*$. Therefore, $f$ is constant.
 \end{myproof}

\subsection{Proof of Lemma~\ref{lemma:multivariate_sum}}\label{app:2}

According to the hypothesis, the derivative of $f(z_1,\cdots,z_p)$ with respect to at least one variable $z_i$ is not constant. 
  Let us take the derivative in both sides of \eqref{eq:statement_lemma_function_h} with respect to an arbitrary variable $z_i$ such that the derivative is not constant. Then, we have:
\begin{multline}\label{eq:derivative_multivariate}
  \!\!\!\!\!\!\prt{\frac{\partial f_i}{\partial z_i}}\!\!(x_1\!+\!g(\y_1,\ldots,\y_r),\ldots,x_p\!+\!g(\tau^{p-1}\y_1,\ldots,\tau^{p-1}\y_r))=\\
  \prt{\frac{\partial f_i}{\partial z_i}}\!\!(x_1\!+\!\tilde g(\y_1,\ldots,\y_r),\ldots,x_p\!+\!\tilde g(\tau^{p-1}\y_1,\ldots,\tau^{p-1}\y_r)).   
  \end{multline}
   Since the partial derivative is continuously differentiable according to the hypothesis,  \eqref{eq:derivative_multivariate} is equivalent to \eqref{eq:equality_multivariate} and we can apply Lemma~\ref{lemma:multivariate} to guarantee that either $g=\tilde g$ or $\prt{\frac{\partial f_i}{\partial z_i}}$ is constant. But if $\prt{\frac{\partial f_i}{\partial z_i}}$ is constant, we have a contradiction and then we can conclude that $g=\tilde g$.

   \begin{IEEEbiography}[{\includegraphics[width=1in,height=1.25in,clip,keepaspectratio]{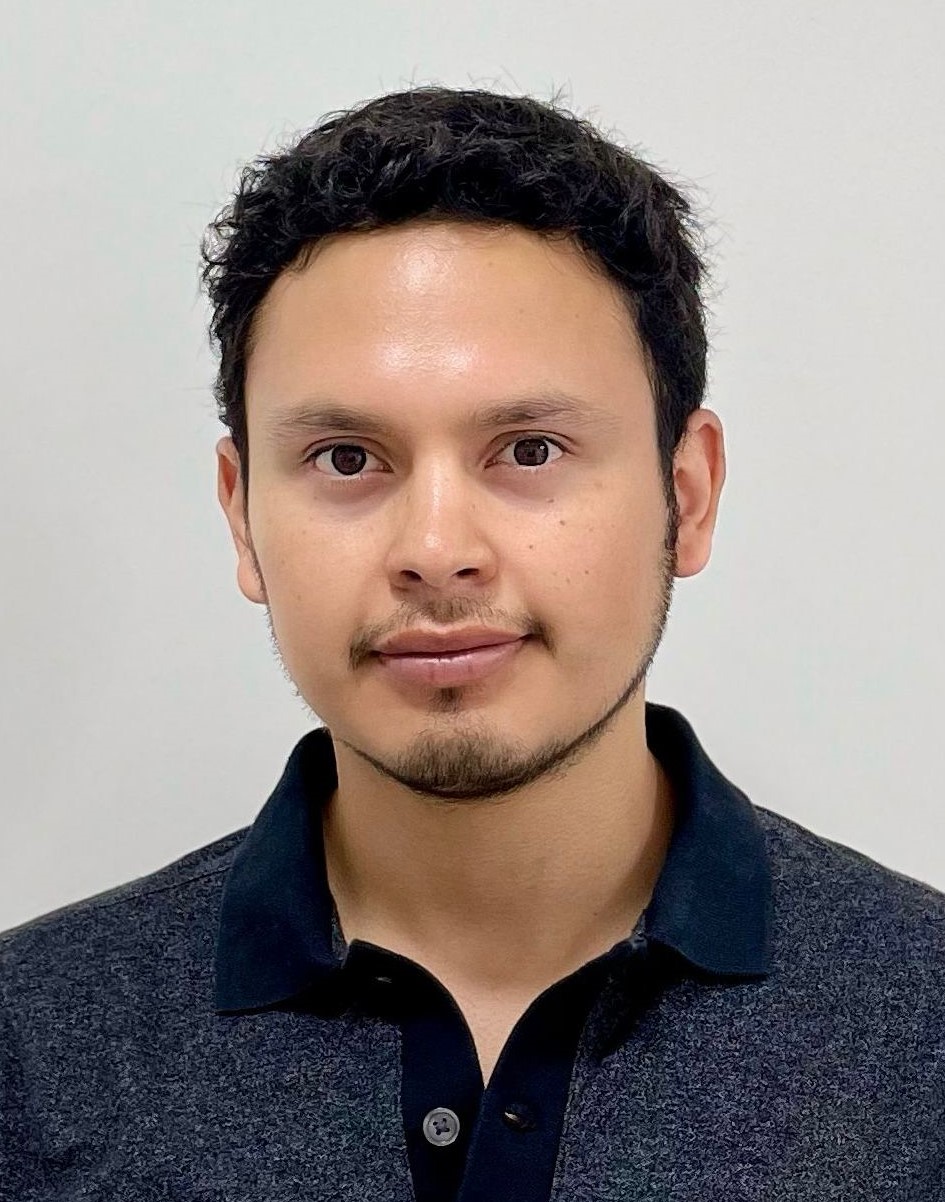}}]{Renato Vizuete}
received the M.S. degree (très bien) in Systems, Control, and Information Technologies from Université Grenoble Alpes, France (2019), and the PhD degree in Automatic Control from Université Paris-Saclay, France (2022). He is currently a FNRS postdoctoral researcher - CR at UCLouvain, Belgium, in the ICTEAM Institute. His research interests include optimization, control theory, machine learning, network theory and system identification. He was the recipient of the Networks and Communication Systems TC Outstanding Student Paper Prize of the IEEE Control Systems Society in 2022, and the Second Thesis Prize in the category Impact Science of the Fondation CentraleSupélec in 2023.
\end{IEEEbiography}

\begin{IEEEbiography}[{\includegraphics[width=1in,height=1.25in,clip,keepaspectratio]{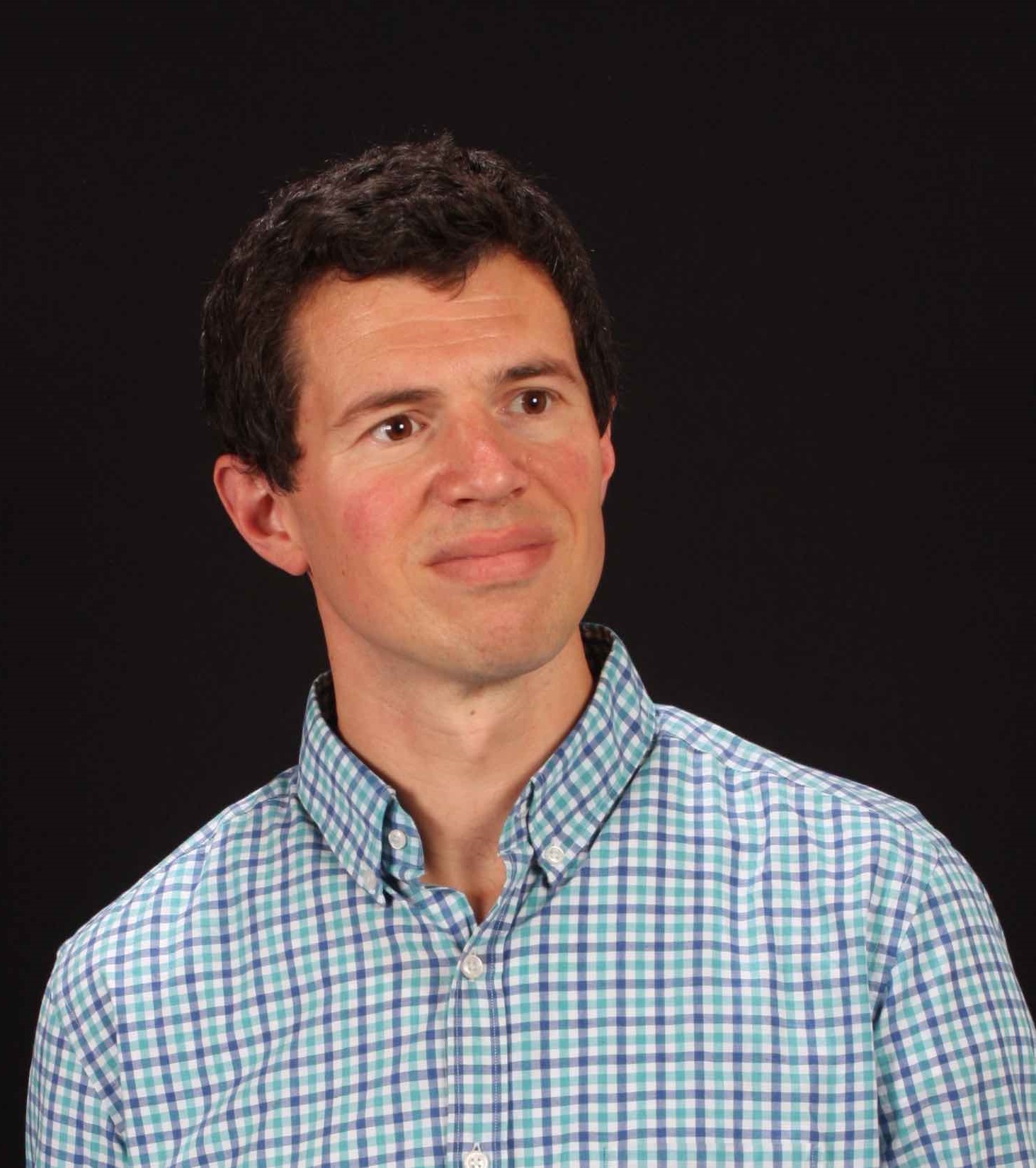}}]{Julien M. Hendrickx}
is professor of mathematical engineering at UCLouvain, in the Ecole Polytechnique de Louvain since 2010, and is currently the head of the ICTEAM institute.

He obtained an engineering degree in applied mathematics (2004) and a PhD in mathematical engineering (2008) from the same university. He has been a visiting researcher at the University of Illinois at Urbana Champaign in 2003-2004, at the National ICT Australia in 2005 and 2006, and at the Massachusetts Institute of Technology in 2006 and 2008. He was a postdoctoral fellow at the Laboratory for Information and Decision Systems of the Massachusetts Institute of Technology 2009 and 2010, holding postdoctoral fellowships of the F.R.S.-FNRS (Fund for Scientific Research) and of Belgian American Education Foundation. He was also resident scholar at the Center for Information and Systems Engineering (Boston University) in 2018-2019, holding a WBI.World excellence fellowship.

Doctor Hendrickx is the recipient of the 2008 EECI award for the best PhD thesis in Europe in the field of Embedded and Networked Control, and of the Alcatel-Lucent-Bell 2009 award for a PhD thesis on original new concepts or application in the domain of information or communication technologies. 
\end{IEEEbiography}

\end{document}